\documentclass[12pt,reqno]{amsproc}%
\usepackage{amsfonts}
\usepackage{amsmath}
\usepackage{amssymb}
\usepackage{graphicx}%
 \theoremstyle{plain}

\newtheorem{definition}{Definition}[section]
\newtheorem{example}{Example}[section]

\newtheorem{sublemma}{Sublemma}[subsection]
\newtheorem{lemma}{Lemma}[section]
\newtheorem{notation}{Notation}[section]

\newtheorem{proposition}{Proposition}[section]
\newtheorem{remark}{Remark}[section]

\newtheorem{theorem}{Theorem}[section]
 \numberwithin{equation}{subsection}
\begin{document}

\centerline{\Large\bf{Topological Free  Entropy Dimension in Unital
C$^*$ algebras }}

\vspace{1cm}

\centerline{\Large\bf{II  : Orthogonal Sum of Unital
C$^*$-algebras}}

\vspace{1cm}

\centerline{Don Hadwin \qquad \qquad  and  \qquad \qquad Junhao
Shen\footnote{The second author is partially supported   by an NSF
grant.}}
\bigskip

\centerline{\small{Department of Mathematics and Statistics,
University of New Hampshire, Durham, NH, 03824}}

\vspace{0.2cm}

\centerline{Email: don@math.unh.edu  \qquad \qquad and \qquad \qquad
jog2@cisunix.unh.edu \qquad }

\bigskip

\noindent\textbf{Abstract: } In the paper,  we   obtain a formula
for topological free entropy dimension in the orthogonal sum (or
direct sum) of unital C$^*$ algebras. As an application, we compute the
topological free entropy dimension of any family of self-adjoint
generators of a finite dimensional C$^*$ algebra.

\vspace{0.2cm} \noindent{\bf Keywords:} Topological free entropy
dimension, C$^*$ algebra

\vspace{0.2cm} \noindent{\bf 2000 Mathematics Subject
Classification:} Primary 46L10, Secondary 46L54

\section{Introduction}The theory of free probability and free entropy was developed by
Voiculescu from 1980s. It played a crucial role in the recent study
of finite von Neumann algebras (see \cite{BDK}, \cite{Dyk}, \cite{DJSh},
\cite{Ge1}, \cite{Ge2}, \cite{GePo}, \cite{GS1}, \cite{GS2},
\cite{HaSh}, \cite{Jung2}, \cite{JuSh}, \cite{V2}, \cite{V3},
\cite{V4}). An analogue of free entropy dimension in C$^*$
algebra context, the notion of topological free entropy dimension of
of $n$-tuples of elements in a unital C$^*$ algebra,  was
introduced by Voiculescu in \cite{Voi}, where some   basic
properties of free entropy dimension are discussed.

We start our investigation of the properties of topological free
entropy dimension in \cite {HaSh2}, where   we computed the
topological free entropy dimension of a self-adjoint element in a
unital C$^*$ algebra. Some estimation of topological free entropy
dimension in an infinite dimensional, unital, simple C$^*$ algebra
with a unique trace was also obtained in the same paper. In this
article, we will continue our investigation on the properties of
topological free entropy dimension.

First, we compute   the topological free entropy dimension in an
$n\times n$ complex matrix algebra  $\mathcal M_{n}(\Bbb C)$ as
follows (see {\bf Theorem 3.1}):
$$
  \delta_{top}(x_1,\ldots,x_m)=1-\frac
1{n^2},
$$ where $x_1,\ldots,x_m$ is any family of self-adjoint generators
of $\mathcal M_{n}(\Bbb C)$ and $\delta_{top}(x_1,\ldots,x_m)$ is
the Voiculescu's topological free entropy dimension of $x_1,\ldots,
x_m$.

In \cite{Voi}, Voiculescu asked the question whether the equality
$$
\chi_{top}(x_1\oplus y_1,\ldots, x_n\oplus y_n)=
\max\{\chi_{top}(x_1 ,\ldots, x_n ),\chi_{top}(y_1 ,\ldots, y_n )\},
$$ holds when $x_1,\ldots, x_n$ and $y_1,\ldots, y_n$ are
self-adjoint elements in a unital C$^*$ algebras $\mathcal A$, or
$\mathcal B$ respectively, and $\chi_{top}$ is the topological free
entropy defined in \cite{Voi}. Motivated by his question, in the
paper we consider the topological free entropy dimension
in the orthogonal sum of unital C$^*$ algebras. More specifically,
we prove the following result.

\vspace{0.1cm}

\noindent { \bf Theorem 4.2: } Suppose that $\mathcal A$ and
$\mathcal B$ are two unital C$^*$ algebras and $x_1\oplus y_1,
\ldots, x_n\oplus y_n$ is a family of self-adjoint elements that
generate  $\mathcal A\bigoplus \mathcal B$ as a C$^*$-algebra.  Assume
$$
s=\delta_{top}(x_1,\ldots,x_n)  \qquad \text { and } \qquad
t=\delta_{top}(y_1,\ldots,y_n).
$$
\noindent (1) If $s\ge 1$ or $t\ge 1$, then
$$
\delta_{top}(x_1\oplus y_1, \ldots, x_n\oplus y_n)=\max
\{\delta_{top}(x_1,\ldots,x_n), \delta_{top}(y_1,\ldots,y_n)\}
$$
\noindent (2)
  If $s<1$,  $t<1$ and both families $\{x_1,\ldots,x_n\}$, $\{y_1,\ldots,
y_n\}$ are stable (see Definition 4.1) , then (i)
$$
\delta_{top}(x_1\oplus y_1, \ldots, x_n\oplus y_n)=
\frac{st-1}{s+t-2} ;
$$
and (ii) the family of elements $x_1\oplus y_1, \ldots, x_n\oplus
y_n$ is also stable.

\vspace{0.1cm}

Combining the preceding two results, Theorem 3.1 and Theorem 4.2, we
obtain the topological free entropy dimension of any family of
self-adjoint generators in a finite dimensional C$^*$ algebra (see
{\bf Theorem 5.1}): Suppose that $\mathcal A$ is a finite
dimensional C$^*$ algebra and $dim_{\Bbb C}\mathcal A$ is the
complex dimension of $\mathcal A$. Then
$$
  \delta_{top}(x_1,\ldots,x_n)= 1- \frac 1 {dim_{\Bbb C}\mathcal A},
$$ where $x_1,\ldots,x_n$ is a family of self-adjoint generators
of $\mathcal A$.

The organization of the paper is as follows. In section 2, we recall
Voiculescu's definition of topological free entropy dimension. The
computation of topological free entropy dimension in an $n\times n$
complex matrix algebra is carried out in section 3. In section 4, we
prove a formula of the topological free entropy dimension in the
orthogonal sum of the unital C$^*$ algebras. In section 5, we
calculate the topological free entropy dimension in any finite
dimensional C$^*$ algebra.

In this article, we only discuss unital C$^*$ algebras which have
the approximation property (see  Definition 5.3  in \cite {HaSh2}).
\section{Definitions  and preliminary}

In this section, we will recall Voiculescu's definition of the
topological free entropy dimension of $n$-tuples of elements in a
unital C$^*$ algebra.
\subsection{A Covering    of  a set  in a metric space}

Suppose $(X,d)$ is a metric space and $K$ is a subset of $X$. A
family of balls in $X$  is called a covering of $K$ if the union of
these balls covers $K$ and the centers of these balls are in $K$.

\subsection{Covering numbers in  complex matrix algebra
$(\mathcal{M}_{k}(\mathbb{C}))^n$}
 Let $\mathcal{M}_{k}(\mathbb{C})$ be the $k\times k$ full matrix
algebra with entries in $\mathbb{C}$,  and $\tau_{k}$ be the
  normalized trace on $\mathcal{M}_{k}(\mathbb{C})$, i.e.,
  $\tau_{k}=\frac{1}{k}Tr$, where $Tr$ is the usual trace on
  $\mathcal{M}_{k}(\mathbb{C})$.
 Let $\mathcal{U}(k)$ denote the group
of all unitary matrices in $\mathcal{M}_{k}(\mathbb{C})$. Let
$\mathcal{M}_{k}(\mathbb{C})^{n}$ denote the direct sum of $n$
copies of $\mathcal{M}_{k}(\mathbb{C})$.  Let $\mathcal
M_k^{s.a}(\Bbb C)$ be the subalgebra of $\mathcal M_k(\Bbb C)$
consisting of all self-adjoint matrices of $\mathcal M_k(\Bbb C)$.
Let $(\mathcal M_k^{s.a}(\Bbb C))^n$ be the direct sum (or
orthogonal sum) of $n$ copies of $\mathcal M_k^{s.a}(\Bbb C)$. Let
$\|\cdot \|$ be an operator norm  on
  $\mathcal{M}_{k}(\mathbb{C})^{n}$ defined by
  \[
  \Vert(A_{1},\ldots,A_{n})\Vert =  \max\{\|A_1\|,\ldots, \|A_n\| \}    \]
  for all $(A_{1},\ldots,A_{n})$ in $\mathcal{M}_{k}(\mathbb{C})^{n}$.
Let $\Vert\cdot\Vert_{Tr}$
  denote the usual trace norm induced by $Tr$ on
  $\mathcal{M}_{k}(\mathbb{C})^{n}$, i.e.,
  \[
  \Vert(A_{1},\ldots,A_{n})\Vert_{Tr} =\sqrt{Tr(A_{1}^{\ast}A_{1})+\ldots
  +Tr(A_{n}^{\ast}A_{n})}
  \]
  for all $(A_{1},\ldots,A_{n})$ in $\mathcal{M}_{k}(\mathbb{C})^{n}$.
  Let $\Vert\cdot\Vert_{2}$
  denote the  trace norm induced by $\tau_{k}$ on
  $\mathcal{M}_{k}(\mathbb{C})^{n}$, i.e.,
  \[
  \Vert(A_{1},\ldots,A_{n})\Vert_{2} =\sqrt{\tau_{k}(A_{1}^{\ast}A_{1})+\ldots
  +\tau_{k}(A_{n}^{\ast}A_{n})}
  \]
  for all $(A_{1},\ldots,A_{n})$ in $\mathcal{M}_{k}(\mathbb{C})^{n}$.

For every $\omega>0$, we define the $\omega$-$\|\cdot\|$-ball
$Ball_{}(B_{1},\ldots ,B_{n};\omega,  \|\cdot\|)$ centered at
$(B_{1},\ldots,B_{n})$ in $\mathcal{M}_{k}(\mathbb{C})^{n}$ to be
the subset of $\mathcal{M}_{k}(\mathbb{C})^{n}$   consisting of all
$(A_{1},\ldots,A_{n})$ in $\mathcal{M}_{k}(\mathbb{C})^{n}$ such
that $$\Vert(A_{1},\ldots,A_{n})-(B_{1},\ldots,B_{n})\Vert
<\omega.$$

\begin{definition}
Suppose that $\Sigma$ is a subset of
$\mathcal{M}_{k}(\mathbb{C})^{n}$. We define $\nu_\infty(\Sigma,
\omega)$ to be the minimal number of $\omega$-$\|\cdot\|$-balls that
consist a covering of $\Sigma$ in $\mathcal{M}_{k}(\mathbb{C})^{n}$.
\end{definition}

For every $\omega>0$, we define the $\omega$-$\|\cdot\|_2$-ball
$Ball_{}(B_{1},\ldots ,B_{n};\omega,  \|\cdot\|_2)$ centered at
$(B_{1},\ldots,B_{n})$ in $\mathcal{M}_{k}(\mathbb{C})^{n}$ to be
the subset of $\mathcal{M}_{k}(\mathbb{C})^{n}$   consisting of all
$(A_{1},\ldots,A_{n})$ in $\mathcal{M}_{k}(\mathbb{C})^{n}$ such
that $$\Vert(A_{1},\ldots,A_{n})-(B_{1},\ldots,B_{n})\Vert_2
<\omega.$$

\begin{definition}
Suppose that $\Sigma$ is a subset of
$\mathcal{M}_{k}(\mathbb{C})^{n}$. We define $\nu_2(\Sigma, \omega)$
to be the minimal number of $\omega$-$\|\cdot\|_2$-balls that
consist a covering of $\Sigma$ in $\mathcal{M}_{k}(\mathbb{C})^{n}$.
\end{definition}

The following lemma is obvious.

\begin{lemma}
Suppose  $K$ is a subset of  $ \mathcal{M}_{k}(\mathbb{C} )^n$, equipped with a distance $d$. Suppose that
$\{B_\lambda\}_{\lambda\in\Lambda}$ is a family of balls of radius $\omega$ so that
$$K\subseteq \cup_{\lambda\in\Lambda}B_\lambda.$$
Then
$$
\text {Covering number of $K$ by balls of radius $2\omega$} \le \text{Cardinality of }\Lambda.
$$
\end{lemma}
% For every $R>0$, we define $(\mathcal M_k(\Bbb C)^n)_R$ to be the
%subset of $\mathcal M_k(\Bbb C)^n$ consisting of all these
%$(A_1,\ldots,A_n)$ in $\mathcal M_k(\Bbb C)^n$ such that $\max_{1\le
%j\le n}\| A_j\|\le R$.
\subsection{Noncommutative polynomials}
In this article, we always assume that $\mathcal A$ is a unital
C$^*$-algebra. Let $x_1,\ldots, x_n, y_1,\ldots, y_m$ be
self-adjoint elements in $\mathcal A$. Let $\Bbb C\langle
X_1,\ldots, X_n, Y_1,\ldots,Y_m\rangle $ be the unital
noncommutative polynomials in the indeterminates $X_1,\ldots, X_n,
Y_1,\ldots,Y_m$. Let $\{P_r\}_{r=1}^\infty$ be the collection of all
noncommutative polynomials in $\Bbb C\langle X_1,\ldots, X_n,
Y_1,\ldots,Y_m\rangle $ with rational complex coefficients. (Here
``rational complex coefficients" means that the real   and imaginary
parts of all coefficients of $P_r$ are rational numbers).

\begin{remark}
   We alsways assume that $1\in  \Bbb C\langle
X_1,\ldots, X_n, Y_1,\ldots,Y_m\rangle $.
\end{remark}

\subsection{Voiculescu's Norm-microstates Space}
For all integers $r, k\ge 1$, real numbers $R, \epsilon>0$ and
noncommutative polynomials $P_1,\ldots, P_r$, we define
$$
\Gamma^{(top)}_R(x_1,\ldots, x_n, y_1,\ldots, y_m; k,\epsilon,
P_1,\ldots, P_r)
$$ to be the subset of $(\mathcal M_k^{s.a}(\Bbb C))^{n+m}$
consisting of all these $$ (A_1,\ldots, A_n, B_1,\ldots, B_m)\in
(\mathcal M_k^{s.a}(\Bbb C))^{n+m}
$$ satisfying
 $$  \max \{\|A_1\|, \ldots, \|A_n\|, \|B_1\|,\ldots, \|B_m\|\}\le R
 $$ and
$$
|\|P_j(A_1,\ldots, A_n,B_1,\ldots,B_m)\|-\| P_j(x_1,\ldots,
x_n,y_1,\ldots,y_m)\||\le \epsilon, \qquad \forall \ 1\le j\le r.
$$

%\begin{remark}
%In the definition of norm-microstates space, we use the following
%assumption. If
%$$
%P_j(x_1,\ldots, x_n,y_1,\ldots,y_m)= \alpha_0\cdot I_{\mathcal A}+
%\sum_{s=1}^N\sum_{1\le i_1,\ldots, i_s\le n+m} \alpha_{i_1\cdots
%i_s}z_{i_1}\cdots z_{i_s}
%$$ where $z_1,\ldots, z_{n+m}$ denotes $x_1,\ldots, x_n, y_1,\ldots,
%y_m$ and $\alpha_0,\alpha_{i_1\cdots i_s}$ are in $\Bbb C$, then
%$$
%P_j(A_1,\ldots, A_n,B_1,\ldots,B_m)= \alpha_0\cdot I_k+
%\sum_{s=1}^N\sum_{1\le i_1,\ldots, i_s\le n+m} \alpha_{i_1\cdots
%i_s}Z_{i_1}\cdots Z_{i_s}
%$$ where $Z_1,\ldots, Z_{n+m}$ denotes $A_1,\ldots, A_n, B_1,\ldots,
%B_m$ and $I_k$ is the identity matrix in $\mathcal M_k(\Bbb C)$.
%\end{remark}

%\begin{remark} In the original definition of norm-microstates space
%in \cite{Voi}, the parameter $R$ was not introduced. Note the
%following  observation:   Let $R> \max\{\|x_1\|,\ldots, \|x_n\|,\|
%y_1\|,\ldots, \|y_m\|\}$. When $r$ is large enough and $\epsilon$ is
%small enough,
%$$
%\Gamma^{(top)}_{R}(x_1,\ldots, x_n, y_1,\ldots, y_m; k,\epsilon,
%P_1,\ldots, P_r)=\Gamma^{(top)}_{\infty}(x_1,\ldots, x_n,
%y_1,\ldots, y_m; k,\epsilon, P_1,\ldots, P_r)
%$$ for all $k\ge 1$.
%Our definition agrees with the one in \cite{Voi} for large $R$, $r$
%and small $\epsilon$.
%\end{remark}

Define the norm-microstates space of $x_1,\ldots,x_n$ in the
presence of $y_1,\ldots, y_m$, denoted by
$$
\Gamma^{(top)}_R(x_1,\ldots, x_n: y_1,\ldots, y_m; k,\epsilon,
P_1,\ldots, P_r)
$$ as the projection of $
\Gamma^{(top)}_R(x_1,\ldots, x_n, y_1,\ldots, y_m; k,\epsilon,
P_1,\ldots, P_r) $ onto the space $(\mathcal M_k^{s.a}(\Bbb C))^n$
via the mapping
$$
(A_1,\ldots, A_n, B_1,\ldots, B_m) \rightarrow (A_1,\ldots, A_n).
$$

\subsection{Voiculescu's topological free entropy dimension}

Define
$$
\nu_\infty(\Gamma^{(top)}_R(x_1,\ldots, x_n: y_1,\ldots, y_m;
k,\epsilon, P_1,\ldots, P_r),\omega)
$$ to be the covering number of the set $
\Gamma^{(top)}_R(x_1,\ldots, x_n: y_1,\ldots, y_m; k,\epsilon,
P_1,\ldots, P_r) $ by $\omega$-$\| \cdot\|$-balls in the metric
space $(\mathcal M_k^{s.a}(\Bbb C))^n$ equipped with operator norm.

\begin{definition}Define
$$
  \begin{aligned}
     \delta_{top}(x_1,\ldots, & \ x_n: y_1,\ldots, y_m;  \omega)\\
     & =
     \sup_{R>0} \ \inf_{\epsilon>0, r\in \Bbb N}  \ \limsup_{k\rightarrow\infty} \frac {\log(\nu_\infty(\Gamma^{(top)}_R(x_1,
     \ldots, x_n: y_1,\ldots, y_m; k,\epsilon, P_1,\ldots,
P_r),\omega))}{-k^2\log\omega}\\ \end{aligned}
$$

\noindent {\bf The topological entropy dimension } of $x_1,\ldots,
x_n$ in the presence of $y_1,\ldots, y_m$ is defined by
$$ \delta_{top}(x_1,\ldots, \ x_n: y_1,\ldots, y_m)=
 \limsup_{\omega\rightarrow 0^+}  \delta_{top}(x_1,\ldots,   x_n:y_1,\ldots, y_m;  \omega)
$$
\end{definition}
\begin{remark} Let $R>
\max\{\|x_1\|,\ldots,\|x_n\|,\|y_1\|,\ldots,\|y_m\|\}$ be some
positive number. By definition, we know
$$
  \begin{aligned}
     \delta_{top}(x_1,\ldots, & \ x_n: y_1,\ldots, y_m )\\
     & =
       \limsup_{\omega\rightarrow 0^+}  \inf_{\epsilon>0, r\in \Bbb N}  \ \limsup_{k\rightarrow\infty}
       \frac {\log(\nu_\infty(\Gamma^{(top)}_R(x_1,
     \ldots, x_n: y_1,\ldots, y_m; k,\epsilon, P_1,\ldots,
P_r),\omega))}{-k^2\log\omega}\\ \end{aligned}
$$
\end{remark}

\begin{remark} Apparently, $\delta_{top}(x_1,\ldots,  x_n: y_1,\ldots, y_m )$ does not depend on the order of
the sequence $\{P_r\}_{r=1}^\infty$.
\end{remark}
\subsection{C$^*$ algebra ultraproduct  and  von Neumann
algebra ultraproduct}

Suppose $\{\mathcal M_{k_m}(\Bbb C)\}_{m=1}^\infty$ is a sequence of
complex matrix algebras where $k_m$ goes to infinity as $m$ goes to
infinity. Let $\gamma$ be a free ultrafilter in $\beta(\Bbb
N)\setminus \Bbb N$. We can introduce a unital C$^*$ algebra
$\prod_{m=1}^\infty \mathcal M_{k_m}(\Bbb C)$ as follows:
$$
\prod_{m=1}^\infty \mathcal M_{k_m}(\Bbb C) = \{(Y_m)_{m=1}^\infty \
| \ \forall \ m\ge 1, \ Y_m \in \mathcal M_{k_m}(\Bbb C) \ \text {
and } \ \sup_{m\ge 1} \| Y_m\|<\infty\}.
$$
We can also introduce   norm   closed two sided ideals $\mathcal
I_\infty$ and $\mathcal I_2$   as follows.
$$
  \begin{aligned}
    \mathcal I_\infty & = \{(Y_m)_{m=1}^\infty\in  \prod_{m=1}^\infty \mathcal M_{k_m}(\Bbb C) \
| \ \lim_{m\rightarrow \gamma } \|Y_m\| =0\}\\
 \mathcal I_2 & = \{(Y_m)_{m=1}^\infty\in  \prod_{m=1}^\infty \mathcal M_{k_m}(\Bbb C) \
| \ \lim_{m\rightarrow \gamma } \|Y_m\|_2 =0\}
  \end{aligned}
$$

\begin{definition}
The C$^*$ algebra ultraproduct of $\{\mathcal M_{k_m}(\Bbb
C)\}_{m=1}^\infty$ along the ultrfilter $\gamma$, denoted by
$\prod_{m=1}^\gamma \mathcal M_{k_m}(\Bbb C)$, is defined to be the
quotient
  algebra of $\prod_{m=1}^\infty \mathcal M_{k_m}(\Bbb C)$ by
the ideal $\mathcal I_\infty$. The image of $(Y_m)_{m=1}^\infty\in
\prod_{m=1}^\infty \mathcal M_{k_m}(\Bbb C)$ in the quotient algebra
is denoted by $[(Y_m)_{m}]$.
\end{definition}

\begin{definition}
The von Neumann algebra ultraproduct of $\{\mathcal M_{k_m}(\Bbb
C)\}_{m=1}^\infty$ along the ultrfilter $\gamma$, also denoted by
$\prod_{m=1}^\gamma \mathcal M_{k_m}(\Bbb C)$ if   no confusion
arises, is defined to be the quotient    algebra of
$\prod_{m=1}^\infty \mathcal M_{k_m}(\Bbb C)$ by the ideal $\mathcal
I_2$. The image of $(Y_m)_{m=1}^\infty\in \prod_{m=1}^\infty
\mathcal M_{k_m}(\Bbb C)$ in the quotient algebra is denoted by
$[(Y_m)_{m}]$.
\end{definition}
\begin{remark}
The von Neumann algebra ultraproduct   $\prod_{m=1}^\gamma \mathcal
M_{k_m}(\Bbb C)$  is a finite  factor (see \cite{McDuff}).
\end{remark}

\section{Topological free entropy dimension in $\mathcal M_n(\Bbb C)$}

In this section, we are going to calculate the topological free
entropy dimension of a family of self-adjoint generators of
$\mathcal M_n(\Bbb C)$.

\subsection{Upper-bound}
\begin{proposition}
   Let $n$ be a positive integer and $\mathcal M_n(\Bbb C)$ be the
   $n\times n$
   matrix algebra over the complex numbers. Let $x_1,\ldots, x_m$ be
   a family of self-adjoint matrices that generate $\mathcal M_n(\Bbb
   C)$. Then
   $$
    \delta_{top}(x_1,\ldots,x_m) \le 1- \frac 1 {n^2}.
   $$
\end{proposition}
\begin{proof} Since $\mathcal M_n(\Bbb C)$ is  a unital C$^*$ algebra
with a unique tracial state, by Theorem 5.1 in \cite {HaSh2}, we
know that
   $$
       \delta_{top}(x_1,\ldots,x_m)\le \kappa\delta(x_1,\ldots,x_m),
   $$
where $\kappa\delta(x_1,\ldots,x_m)$ is the Voiculescu's free
dimension capacity in \cite{Voi}. By \cite{Jung3} or Proposition 1
in \cite{HaSh}, we have
$$  \kappa\delta(x_1,\ldots,x_m)\le 1- \frac 1 {n^2}.$$
Therefore,
$$
 \delta_{top}(x_1,\ldots,x_m) \le 1- \frac 1 {n^2}.
$$
\end{proof}

\subsection{Some lammas}

In this subsection, we let $n,t$ be some positive integers and
$k=nt$. Let
$$
A =\left (
   \begin{aligned}
     1\cdot I_t \quad & 0 \quad & \quad \cdots \quad &\quad
   0\\
   0 \quad & 2\cdot I_t \quad & \quad \cdots \quad &\quad
     0\\
     \cdots \quad & \quad \cdots \quad & \quad \cdots \quad &\quad
     \cdots\\
     0 \quad & \quad 0 \quad & \quad \cdots \quad &\quad
     n\cdot I_t\\
   \end{aligned}
 \right ) \quad \text { and } \quad W= \left (
   \begin{aligned}
     0\quad & 1\cdot I_t  \quad & \quad 0\quad &\quad \cdots \quad &\quad
   0\\
   0 \quad & \quad  0\quad &1\cdot I_t \quad & \quad \cdots \quad &\quad
     0\\
     \cdots \quad & \quad \cdots \quad& \quad \ddots \quad & \quad \cdots \quad &\quad
     \cdots\\
     1 \cdot I_t& \quad 0 \quad &\quad 0 \quad & \quad \cdots \quad &\quad
    0\\
   \end{aligned}
 \right )
$$ be in $\mathcal M_k(\Bbb C)$, where $I_t$ is the identity matrix of $\mathcal M_t(\Bbb C)$.

\begin{lemma} Let $\delta>0$. Suppose   $\|U_1AU_1^*- U_2AU_2^*\|_2\le   \delta  $   and  $\|U_1WU_1^*- U_2WU_2^*\|_2\le   \delta  $ for some
unitary matrices $U_1$ and $U_2$ in $\mathcal U(k)$. Then  there are
  a unitary matrix $V_1$ in $\mathcal M_t(\Bbb C)$ and
$$
     V=\left (
   \begin{aligned}
     V_1\quad &\quad  0 \quad   &\quad \cdots \quad &\quad
   0\\
   0 \quad & V_1   & \quad \cdots \quad &\quad
     0\\
     \cdots \quad & \quad \cdots \quad  & \quad \ddots \quad &\quad
     \cdots\\
   0 \quad & \quad 0 \quad   & \quad \cdots \quad &V_1\\
   \end{aligned}
 \right ) \in   \ \mathcal U(k)
$$
so that
$$
 \|U_1-U_2V\|_2\le {14n^2\delta}.
$$
\end{lemma}
\begin{proof}
Assume that
$$U_2^*U_1 =\left (
   \begin{aligned}
     U_{11} \quad & \quad U_{12} \quad & \quad \cdots \quad &\quad
     U_{1, n}\\
     U_{21} \quad & \quad U_{22} \quad & \quad \cdots \quad &\quad
     U_{2, n}\\
     \cdots \quad & \quad \cdots \quad & \quad \cdots \quad &\quad
     \cdots\\
     U_{n,1} \quad & \quad U_{n,2} \quad & \quad \cdots \quad &\quad
     U_{n,n}\\
   \end{aligned}
 \right )  \in   \ \mathcal U(k) $$ where each
$U_{i,j}$ is a $t\times t$ matrix for all $1\le i,j\le n$.

 Let$$ S =\left ( \begin{aligned}
     U_{11} \quad & \quad 0 \quad & \quad \cdots \quad &\quad 0 \quad   \\
     0 \quad & \quad U_{22} \quad & \quad \cdots \quad  &\quad 0 \quad  \\
     \cdots \quad & \quad \cdots \quad & \quad \ddots \quad & \quad \cdots \quad \\
    0 \quad & \quad 0 \quad & \quad \cdots \quad &\quad U_{n ,n}
   \end{aligned}\right).
$$ It is easy to see that  $\| S\|_2\le 1$ and
$$
   \begin{aligned}
    \delta^2 &\ge    \|U_1AU_1^*- U_2AU_2^*\|_2^2 = \frac 1 k Tr((U_2^*U_1A-
    AU_2^*U_1)(U_2^*U_1A-
    AU_2^*U_1)^*)\\
    &=\frac 1 k \sum_{1\le i\ne j\le m}Tr(| i- j|^2U_{ij}U_{ij}^*)\\
    &\ge \frac 1 k   \sum_{1\le i\ne j\le m} Tr(U_{ij}U_{ij}^*).
   \end{aligned}
$$
 Hence
\begin{equation}\| U_1- U_2S \|_2 =\| U_2^*U_1-S \|_2=\sqrt{\frac 1 k \sum_{1\le i\ne j\le m}  Tr(U_{ij}U_{ij}^*)}\le   {  \delta } .
\end{equation}
Thus, $$ \begin{aligned} &\left\| \left ( \begin{aligned}
     U_{22} \quad & \quad 0 \quad & \quad \cdots \quad &\quad 0 \quad   \\
     0 \quad & \quad U_{33} \quad & \quad \cdots \quad  &\quad 0 \quad  \\
     \cdots \quad & \quad \cdots \quad & \quad \ddots \quad & \quad \cdots \quad \\
    0 \quad & \quad 0 \quad & \quad \cdots \quad &\quad U_{11}
   \end{aligned}\right)-\left ( \begin{aligned}
     U_{11} \quad & \quad 0 \quad & \quad \cdots \quad &\quad 0 \quad   \\
     0 \quad & \quad U_{22} \quad & \quad \cdots \quad  &\quad 0 \quad  \\
     \cdots \quad & \quad \cdots \quad & \quad \ddots \quad & \quad \cdots \quad \\
    0 \quad & \quad 0 \quad & \quad \cdots \quad &\quad U_{n ,n}
   \end{aligned}\right)\right\|_2\\&\quad \\&\qquad
   =\|W^*SW-S\|_2 =\| SW-WS\|_2=\|U_2^*U_1W-WU_2^*U_1-(U_2^*U_1-S)W+W(U_2^*U_1-S)   \|_2\le 3\delta. \end{aligned} $$
   It follows that
   \begin{equation}
   \frac 1 {\sqrt k} \sqrt{Tr
((U_{j,j}-U_{j+1,j+1})(U_{j,j}-U_{j+1,j+1})^*)}\le 3\delta,\ \ \forall \ 1\le j\le n-1.
   \end{equation}
Let $$ X= \left ( \begin{aligned}
     U_{11} \quad & \quad 0 \quad & \quad \cdots \quad &\quad 0 \quad   \\
     0 \quad & \quad U_{11} \quad & \quad \cdots \quad  &\quad 0 \quad  \\
     \cdots \quad & \quad \cdots \quad & \quad \ddots \quad & \quad \cdots \quad \\
    0 \quad & \quad 0 \quad & \quad \cdots \quad &\quad U_{11}
   \end{aligned}\right) . $$
By inequality (3.2.2), we have
 \begin{align}
\|S-X\|_2 \notag &\le \frac 1 {\sqrt k}\sqrt{\sum_{i=2}^n Tr
((U_{11}-U_{ii})(U_{11}-U_{ii})^*)}\le  \frac 1 {\sqrt k}\sum_{i=2}^n \sqrt{Tr
((U_{11}-U_{ii})(U_{11}-U_{ii})^*)}\notag \\
& \le \frac 1 {\sqrt k}\sum_{i=2}^n \sum_{j=1}^{i-1}\sqrt{Tr
((U_{j,j}-U_{j+1,j+1})(U_{j,j}-U_{j+1,j+1})^*)} < 3n^2\delta.\end{align}
 %Or $$
  %\|S-X\|_2\le 3n^2\delta.
%$$
Let $U_{11}=V_1H$ be the polar decomposition of $U_{11}$ in
$\mathcal M_t(\Bbb C)$ and
$$
  V= \left ( \begin{aligned}
     V_1 \quad & \quad 0 \quad & \quad \cdots \quad &\quad 0 \quad   \\
     0 \quad & \quad V_1 \quad & \quad \cdots \quad  &\quad 0 \quad  \\
     \cdots \quad & \quad \cdots \quad & \quad \ddots \quad & \quad \cdots \quad \\
    0 \quad & \quad 0 \quad & \quad \cdots \quad &\quad V_1
   \end{aligned}\right) .
$$
Note $\|H\|=\|U_{11}\|\le \|S\|\le 1.$ We have
$$\begin{aligned}
\|X-V\|_2&= \|\left ( \begin{aligned}
     H \quad & \quad 0 \quad & \quad \cdots \quad &\quad 0 \quad   \\
     0 \quad & \quad H \quad & \quad \cdots \quad  &\quad 0 \quad  \\
     \cdots \quad & \quad \cdots \quad & \quad \ddots \quad & \quad \cdots \quad \\
    0 \quad & \quad 0 \quad & \quad \cdots \quad &\quad H
   \end{aligned}\right) -I\|_2\\
   &\le \|\left ( \begin{aligned}
     H^2 \quad & \quad 0 \quad & \quad \cdots \quad &\quad 0 \quad   \\
     0 \quad & \quad H^2 \quad & \quad \cdots \quad  &\quad 0 \quad  \\
     \cdots \quad & \quad \cdots \quad & \quad \ddots \quad & \quad \cdots \quad \\
    0 \quad & \quad 0 \quad & \quad \cdots \quad &\quad H^2
   \end{aligned}\right) -I\|_2\\
   &=\|X^*X-I\|_2\le 2\|S-X\|_2+ \|S^*S-I\|_2\le 6n^2\delta+2\delta,
   \end{aligned}
$$  where the last inequality follows from inequalities (3.2.1) and (3.2.3). It follows that
$$
\|U_1-U_2V\|_2\le \|U_1-U_2S\|_2+ \|S-X\|_2+ \|X-V\|_2\le 3\delta +
3n^2\delta+ 6n^2\delta+2\delta\le 14 n^2\delta.
$$

\end{proof}

\begin{lemma}Let $k=nt$  and $$\mathcal N_1= \left\{%\oplus_{i=1}^n B=
\left ( \begin{aligned}
     V_1 \quad & \quad 0 \quad & \quad \cdots \quad &\quad 0 \quad   \\
     0 \quad & \quad  V_1 \quad & \quad \cdots \quad  &\quad 0 \quad  \\
     \cdots \quad & \quad \cdots \quad & \quad \ddots \quad & \quad \cdots \quad \\
    0 \quad & \quad 0 \quad & \quad \cdots \quad &\quad  V_1
   \end{aligned}\right )\in \mathcal U(k) \ | \  V_1 \in  \mathcal U(t)\right \} \subset \mathcal M_k(\Bbb
C). $$
 For every $U\in \mathcal U(k)$, let
$$\begin{aligned}
 \Sigma(U)&= \{U_1\in \mathcal U(k)  \ | \ \exists \text{ a unitary
matrix $V $   in $\mathcal N_1$}  \text{ such that $\|U_1-UV\|_2\le
  { 14 n^2 \delta} $ }\}.\end{aligned}
$$ Then
$$
\mu( \Sigma(U)) \le (C_1\cdot 30n^2\delta )^{k^2}\cdot \left
(\frac{C }{ \delta}\right )^{t^2 },
$$ where $\mu$ is the normalized Haar measure on $\mathcal U(k)$ and $C, C_1$ are     constants independent of $t,\delta$.
\end{lemma}
\begin{proof}  \ By computing the covering number of   $\mathcal N_1$
  by $
\delta $-$\|\cdot\|_2$-balls in $\mathcal U(k)$, we know
$$\begin{aligned}
\nu_2(   \mathcal N_1 ,
\delta )& \le \left (\frac {C }{ \delta}\right )^{t^2},
\end{aligned}
$$ where $C$ is a   constant independent of $t,\delta$. Thus, by Lemma 2.1,  the covering number of the set $\Sigma(U)$ by the $30 n^2 \delta $-$\|\cdot\|_2$-balls
 in $\mathcal U(k)$ is bounded by
$$
\nu_2(\Sigma(U),  {30 n^2 \delta} ) \le \nu_2(  \mathcal N_1 ,  \delta  )\le\left (\frac{C }{
\delta}\right )^{t^2}.
$$
But the ball of radius $30 n^2 \delta $ in $\mathcal U(k)$ has the
volume bounded by
$$
\mu(\text {ball of radius $30 n^2 \delta$ in $\mathcal U(k)$})\le (C_1\cdot 30 n^2
\delta )^{k^2},
$$ where $C_1$ is a universal constant.
Thus
$$
\mu(\Sigma(U)) \le (C_1\cdot 30 n^2 \delta )^{k^2}\cdot  \left
(\frac{C }{ \delta}\right )^{t^2}.
$$
\end{proof}
\begin{lemma}
Let $A$, $W$ and $$\mathcal N_1= \left\{%\oplus_{i=1}^n B=
\left ( \begin{aligned}
     V_1 \quad & \quad 0 \quad & \quad \cdots \quad &\quad 0 \quad   \\
     0 \quad & \quad  V_1 \quad & \quad \cdots \quad  &\quad 0 \quad  \\
     \cdots \quad & \quad \cdots \quad & \quad \ddots \quad & \quad \cdots \quad \\
    0 \quad & \quad 0 \quad & \quad \cdots \quad &\quad  V_1
   \end{aligned}\right )\in \mathcal U(k) \ | \  V_1 \in  \mathcal U(t)\right \} \subset \mathcal M_k(\Bbb
C). $$  be defined as above. Let
$$
  \Omega(A, W) = \{(U^*AU, \  \frac 1 2U^*( {W+W^*} )U, \  \frac 1
  {2\sqrt{-1}}U^*(
  {W-W^*})U) \ | \ U\in \mathcal U(k)\}.
$$
Then, for each $\delta>0$,
$$
\nu_2(\Omega(A,W),\frac 1 4\delta)  \ge(C_1\cdot 30 n^2 \delta
)^{-k^2}\cdot  \left (\frac{C }{ \delta}\right )^{-t^2},
$$ where $C_1, C$ are some universal constants independent of $t,\delta$.
\end{lemma}

\begin{proof}
For every $U\in \mathcal U(k)$, define
$$
\Sigma(U)= \{U_1\in \mathcal U(k)  \ | \ \exists \text{ a unitary
matrix } V\in \mathcal N_1, \text{ such that } \ \|U_1- UV\|_2\le
 14 n^2 \delta \}.
$$
By preceding lemma, we have
$$
\mu(\Sigma(U)) \le (C_1\cdot 30 n^2 \delta )^{k^2}\cdot  \left
(\frac{C }{ \delta}\right )^{t^2}.
$$
 A ``parking" (or  exhausting) argument will show
the existence of a family of unitary elements $\{ U_i
\}_{i=1}^N\subset \mathcal U(k)$ such that
$$
N\ge (C_1\cdot 30 n^2 \delta )^{-k^2}\cdot  \left (\frac{C }{
\delta}\right )^{-t^2}
$$  and
$$
  U_i \  \text { is not contained in } \cup_{j=1}^{i-1}\Sigma (U_j), \qquad \forall \ i=1,\ldots, N.
$$

From the definition of each $\Sigma(U_j)$, it follows that  $$ \|
U_i-U_jV\|_2\ge { {14 n^2 \delta } } ,\qquad \forall \ \text {
unitary matrix } V\in \mathcal N_1,   \forall \ 1\le j<i \le N.
$$
By Lemma 3.1, we know that
$$
\|U_iAU_i^*-U_jAU_j^*\|_2>  \delta   \qquad  \text { or } \qquad
\|U_iWU_i^*-U_jWU_j^*\|_2>  \delta,
$$
which implies that
$$
\nu_2(\Omega(A,W),\frac 1 4\delta) \ge N\ge(C_1\cdot 30 n^2 \delta
)^{-k^2}\cdot  \left (\frac{C }{ \delta}\right )^{-t^2}.
$$
\end{proof}

\subsection{Lower-bound}
Suppose $x_1,\ldots, x_m$ is a family of self-adjoint elements that
generate $\mathcal M_n(\Bbb C)$. Let $\{e_{st}\}_{s,t=1}^n$ be a
canonical system of matrix units in $\mathcal M_n(\Bbb C)$. We might
assume that
$$
x_i= \sum_{s,t=1}^n x_{st}^{(i)} \cdot e_{st}, \qquad \forall \ 1\le
i\le m,
$$ for some $\{x_{st}^{(i)}\}_{1\le s,t\le n, 1\le i\le m} \subset
\Bbb C$. Let $$ a= \sum_{i=1}^n i\cdot e_{ii} \qquad \text { and }
\qquad w= \sum_{i=1}^{n-1}   e_{i,i+1}+   e_{n,1}.
$$
Note that $\mathcal M_n(\Bbb C)$ is a finite dimensional C$^*$
algebra. It is easy to see that there exist noncommutative
polynomials $P_1(x_1,\ldots,x_m)$ and $P_2(x_1,\ldots,x_m)$ such
that
$$
a= P_1(x_1,\ldots,x_m)  \qquad \text { and } \qquad
w=P_2(x_1,\ldots,x_m).
$$

The proof of Lemma 5.1 in \cite {HaSh2} can be easily adapted to
prove the following Lemma 3.4.

\begin{lemma}
We have
$$
\delta_{top}(a, \frac {w+w^*} 2  , \frac  {w-w^*} {2\sqrt{-1}}
:x_1,\ldots,x_m)\le \delta_{top}(x_1,\ldots,x_m).
$$
\end{lemma}

\begin{lemma}
We have
$$
\delta_{top}(a, \frac {w+w^*} 2  , \frac  {w-w^*} {2\sqrt{-1}}
:x_1,\ldots,x_m) \ge 1-\frac 1 {n^2}.
$$

\end{lemma}

\begin{proof}
Let $t$ be positive integer and $k=nt$. Note that
$$
   \begin{aligned}
      x_i&= \sum_{s,t=1}^n x_{st}^{(i)} \cdot e_{st}, \qquad \forall \ 1\le
i\le m \\
a&= \sum_{i=1}^n i\cdot e_{ii}\\ w&= \sum_{i=1}^{n-1}   e_{i,i+1}+
e_{n,1}.
   \end{aligned}
$$ We let
$$\begin{aligned}
X_i &=\left( \sum_{s,t=1}^n x_{st}^{(i)} \cdot e_{st}\right)\otimes
I_t, \qquad \forall \ 1\le
i\le m \\
A&= \left(\sum_{i=1}^n i\cdot e_{ii}\right)\otimes I_t\\ W&=
\left(\sum_{i=1}^{n-1} e_{i,i+1}+ e_{n,1}\right)\otimes
I_t\end{aligned}
$$ be matrices in $\mathcal M_k(\Bbb C)$.
It is not hard to see that, for every $t\in \Bbb N$ and $k=nt$,
$$
(A, \frac {W+W^*} 2  , \frac  {W-W^*} {2\sqrt{-1}}) \in
\Gamma_R^{(top)}(a, \frac {w+w^*} 2  , \frac  {w-w^*} {2\sqrt{-1}}
:x_1,\ldots,x_m; k,\epsilon, P_1,\ldots, P_r)
$$ when $R> \max\{\|a\|, \|x_1\|,\ldots, \|x_m\|, 1\}$,
$\epsilon>0$ and $r\ge 1$. Therefore,
$$
\Omega(A,W) \subset \Gamma_R^{(top)}(a, \frac {w+w^*} 2  , \frac
{w-w^*} {2\sqrt{-1}} :x_1,\ldots,x_m; k,\epsilon, P_1,\ldots, P_r),
$$ where $\Omega(A,W)$ is defined   in Lemma 3.3. Letting $\delta=4\omega$, by lemma 3.3, we
have
$$
\nu_2(\Gamma_R^{(top)}(a, \frac {w+w^*} 2  , \frac  {w-w^*}
{2\sqrt{-1}} :x_1,\ldots,x_m; k,\epsilon, P_1,\ldots,
P_r),\omega)\ge (C_1\cdot 120 n^2 \omega )^{-k^2}\cdot  \left
(\frac{4C }{ \omega}\right )^{-t^2},
$$ where $C_1, C$ are some constants independent of $t,\omega$.
By the definitions of the operator norm and the trace norm on
$(\mathcal M_k(\Bbb C))^{3}$, we get
 \begin{align}
\nu_\infty(\Gamma_R^{(top)} (a, \frac {w+w^*} 2  , &\frac  {w-w^*}
{2\sqrt{-1}} :x_1,\ldots,x_m; k,\epsilon, P_1,\ldots, P_r),\frac\omega{\sqrt{3}})\notag\\
&\ge \nu_2(\Gamma_R^{(top)}(a, \frac {w+w^*} 2  , \frac {w-w^*}
{2\sqrt{-1}} :x_1,\ldots,x_m; k,\epsilon, P_1,\ldots,
P_r),\omega)\notag\\
&\ge (C_1\cdot 120 n^2 \omega )^{-k^2}\cdot  \left (\frac{4C }{
\omega}\right )^{-t^2}.\end{align}  It  quickly induces that
$$
\delta_{top}(a, \frac {w+w^*} 2  , \frac  {w-w^*} {2\sqrt{-1}}
:x_1,\ldots,x_m) \ge 1-\frac 1 {n^2}.
$$
\end{proof}
Combining Lemma 3.4 and Lemma 3.5, we have the following result.
\begin{proposition}
Suppose $x_1,\ldots, x_m$ is a family of self-adjoint generators of
$\mathcal M_n(\Bbb C)$. Then
$$
\delta_{top}( x_1,\ldots,x_m) \ge 1-\frac 1 {n^2}.
$$
\end{proposition}

\subsection{Conclusion}
By Proposition 3.1 and Proposition 3.2, we obtain the following
result.
\begin{theorem}
Suppose $x_1,\ldots, x_m$ is a family of self-adjoint generators of
$\mathcal M_n(\Bbb C)$. Then
$$
\delta_{top}( x_1,\ldots,x_m)= 1-\frac 1 {n^2}.
$$
\end{theorem}

\section{Topological free entropy dimension in orthogonal sum of
C$^*$ algebras}

In this section, we assume that $\mathcal A$, $\mathcal B$ are
two unital C$^*$ algebras and $\mathcal A \bigoplus \mathcal B$ is
the orthogonal sum, or direct sum, of $\mathcal A$ and $\mathcal B$.
We assume that the self-adjoint elements $x_1\oplus y_1, \cdots,
x_n\oplus y_n$ generate $\mathcal A \bigoplus \mathcal B$ as a C$^*$
algebra. Thus $x_1,\ldots, x_n$ and $y_1,\ldots, y_n$, are the
families of self-adjoint generators of $\mathcal A$, or $\mathcal B$
respectively.

\subsection{Upper-bound of topological free entropy dimension in orthogonal sum of
C$^*$ algebras }

Let $R>\max\{\|x_1\|, \ldots,\|x_n\|, \|y_1\|,\ldots,\|y_n\|\}$ be a
positive number. By the definition of topological free entropy
dimension, we have the following.
\begin{lemma}
For each
$$
\alpha>\delta_{top}(x_1,\ldots,x_n) \qquad \text { and } \qquad
\beta>\delta_{top}(y_1,\ldots,y_n),
$$ (i) there is some $\frac 1 {10}>\omega_0>0$ so that, if $0<\omega<\omega_0$,
$$
   \begin{aligned}
      \inf_{r\in\Bbb N} \limsup_{k_1\rightarrow \infty} & \frac {\log(\nu_{\infty}(\Gamma_R^{(top)}(x_1,\ldots,x_n;k_1,\frac 1 r,
      P_1,\ldots,P_r),\omega))}{-k_1^2\log\omega} <\alpha;\\
       \inf_{r\in\Bbb N} \limsup_{k_2\rightarrow \infty} & \frac {\log(\nu_{\infty}(\Gamma_R^{(top)}(y_1,\ldots,y_n;k_2,\frac 1 r,
      P_1,\ldots,P_r),\omega))}{-k_2^2\log\omega} <\beta.
   \end{aligned}
$$
(ii) Thus, for each $0<\omega<\omega_0$, there is $r(\omega)\in\Bbb
N$ satisfying
$$
   \begin{aligned}
       \limsup_{k_1\rightarrow \infty} & \frac {\log(\nu_{\infty}(\Gamma_R^{(top)}(x_1,\ldots,x_n;k_1,\frac 1 {r(\omega)},
      P_1,\ldots,P_{r(\omega)}),\omega))}{-k_1^2\log\omega} <\alpha;\\
        \limsup_{k_2\rightarrow \infty} & \frac {\log(\nu_{\infty}(\Gamma_R^{(top)}(y_1,\ldots,y_n;k_2,\frac 1 {r(\omega)},
      P_1,\ldots,P_{r(\omega)}),\omega))}{-k_2^2\log\omega} <\beta.
   \end{aligned}
$$
(iii) Therefore, for each $0<\omega<\omega_0$ and $r(\omega)\in \Bbb
N$, there is some $K(r(\omega))\in \Bbb N$ satisfying
$$
   \begin{aligned}
       & {\log(\nu_{\infty}(\Gamma_R^{(top)}(x_1,\ldots,x_n;k_1,\frac 1 {r(\omega)},
      P_1,\ldots,P_{r(\omega)}),\omega))} <-\alpha{ k_1^2\log\omega}, \quad \forall \ k_1\ge K( r(\omega));\\
         &   {\log(\nu_{\infty}(\Gamma_R^{(top)}(y_1,\ldots,y_n;k_2,\frac 1 {r(\omega)},
      P_1,\ldots,P_{r(\omega)}),\omega))} <-\beta{k_2^2\log\omega}, \quad \forall \ k_2\ge
      K( r(\omega)).
   \end{aligned}
$$
\end{lemma}

\begin{lemma}
Suppose that $\mathcal A$ and $\mathcal B$ are two unital C$^*$
algebras and $x_1\oplus y_1, \ldots, x_n\oplus y_n$ is a family of
self-adjoint elements that generate  $\mathcal A\bigoplus \mathcal
B$. Let $R>\max\{\|x_1\|, \ldots,\|x_n\|, \|y_1\|,\ldots,\|y_n\|\}$
be a positive number. For any $\omega>0$, $r_0\in \Bbb N$, there is
some $t>0$ so that the following holds: $\forall \ r>t$, $\forall \
k\ge 2$, if
$$
  (X_1,\ldots, X_n) \in \Gamma_R^{(top)}(x_1\oplus y_1,\ldots,x_n\oplus y_n;k,\frac 1 r,
      P_1,\ldots,P_r),
$$ then there are $k_1,k_2\in \Bbb N$,
$$
   \begin{aligned}
       (A_1,\ldots, A_n) &\in \Gamma_R^{(top)}(x_1  ,\ldots,x_n ;k_1,\frac 1 r_0,
      P_1,\ldots,P_{r_0}),\\
        (B_1,\ldots, B_n) &\in \Gamma_R^{(top)}(y_1  ,\ldots,y_n ;k_2,\frac 1 r_0,
      P_1,\ldots,P_{r_0})
   \end{aligned}
$$ and $U\in \mathcal U(k)$ so that (i) $k_1+k_2=k$; and (ii)
$$
  \left \| (X_1,\ldots, X_n) -U^* ( \begin{pmatrix} A_1 &   0\\ 0&B_1\end{pmatrix},\ldots, \begin{pmatrix} A_n &   0\\ 0&B_n\end{pmatrix}    )   U        \right
  \|\le \omega.
$$

\end{lemma}
\begin{proof}
We will prove the result by using the contradiction. Assume, to the
contrary, the result of the lemma does not hold, i.e. there are some
$\omega_0>0$, $r_0\ge 1$, two strictly increasing sequences
$\{r_m\}_{m=1}^\infty$ and $\{k_m\}_{m=1}^\infty,$ and
$$
(X_1^{(m)},\ldots, X_n^{(m)}) \in \Gamma_R^{(top)}(x_1\oplus
y_1,\ldots,x_n\oplus y_n;k_m,\frac 1 {r_m},
      P_1,\ldots,P_{r_m})
$$ satisfying
\begin{equation}
  \left \| (X_1^{(m)},\ldots, X_n^{(m)}) -U^* (  \begin{pmatrix} A_1^{(m)} &   0\\ 0&B_1^{(m)}\end{pmatrix},\ldots, \begin{pmatrix} A_n^{(m)} &   0\\ 0&B_n^{(m)}\end{pmatrix}     )   U        \right
  \|>  \omega.
\end{equation} for all
$$
   \begin{aligned}
       (A_1^{(m)},\ldots, A_n^{(m)}) &\in \Gamma_R^{(top)}(x_1  ,\ldots,x_n ;s_{_{1,m}},\frac 1 {r_0},
      P_1,\ldots,P_{r_0}),\\
        (B_1^{(m)},\ldots, B_n^{(m)}) &\in \Gamma_R^{(top)}(y_1  ,\ldots,y_n ;s_{_{2,m}},\frac 1 {r_0},
      P_1,\ldots,P_{r_0})
   \end{aligned}
$$ and all $U\in \mathcal U(k)$ where  $s_{_{1,m}}+s_{_{2,m}}=k_m$.

Let $\gamma$ be a free ultra-filter in $\beta(\Bbb N)\setminus \Bbb
N$. Let $\prod_{m=1}^\gamma \mathcal M_{k_m}(\Bbb C)$ be the C$^*$
algebra ultra-product of
 $(\mathcal M_{k_m}(\Bbb C))_{m=1}^\infty$ along the
ultra-filter $\gamma$, i.e.  $\prod_{m=1}^\gamma \mathcal
M_{k_m}(\Bbb C)$ is the quotient algebra of the C$^*$ algebra
$\prod_m \mathcal M_{k_m}(\Bbb C)$ by $\mathcal I_\infty $, the
$0$-ideal of the norm $\|\cdot \|_\gamma$, where
$\|(Y_m)_{m=1}^\infty\|_\gamma=\lim_{m\rightarrow \gamma} \|Y_m\| $
for each $(Y_m)_{m=1}^\infty $ in $\prod_m \mathcal M_{k_m}(\Bbb
C)$.

By mapping  $x_i\oplus y_i$ to $[(X_i^{(m)})_{m=1}^\infty]$ in
$\prod_{m=1}^\gamma \mathcal M_{k_m}(\Bbb C)$  for each $1\le i\le
n$, we obtain a unital $*$-isomorphism $\psi$ from the C$^*$ algebra
$\mathcal A\bigoplus\mathcal B$ onto the C$^*$ subalgebra generated
by $[(X_1^{(m)})_{m=1}^\infty]$, \ldots,
$[(X_n^{(m)})_{m=1}^\infty]$ in  $\prod_{m=1}^\gamma \mathcal
M_{k_m}(\Bbb C)$ . Thus $\psi(I_{\mathcal A}\oplus 0)$ and
$\psi(0\oplus I_{\mathcal B})$ are two projections in
$\prod_{m=1}^\gamma \mathcal M_{k_m}(\Bbb C)$ satisfying
$$\psi(I_{\mathcal A}\oplus 0)+\psi(0\oplus I_{\mathcal
B})=I_{\prod_{m=1}^\gamma \mathcal M_{k_m}(\Bbb C)}.$$

Without loss of generality, we can assume that there is a sequence
of projections $\{P_m\}_{m=1}^\infty$  with $P_m\in \mathcal
M_{k_m}(\Bbb C)$ such that
$$
[(P_m)_{m=1}^\infty] = \psi(I_{\mathcal A}\oplus 0)\qquad \text {
and } \qquad [(I_{k_m}-P_m)_{m=1}^\infty]= \psi(0\oplus I_{\mathcal
B}),
$$ where $I_{k_m}$ is the identity matrix of $\mathcal M_{k_m}(\Bbb
C)$. For each $P_m$ in $\mathcal M_{k_m}(\Bbb C)$, there are
positive integers $s_{_{1,m}}, s_{_{2,m}}$, with
$s_{_{1,m}}+s_{_{2,m}}=k_m$, and a unitary matrix $U_m$ in $\mathcal
U(k_m)$ so that
$$
 P_m= U_m^* \begin{pmatrix} I_{s_{_{1,m}}} &   0\\ 0&0\end{pmatrix} U_m \qquad \text { and } \qquad I_{k_m}-P_m=
 U_m^*\begin{pmatrix} 0 &   0\\ 0&I_{s_{_{2,m}}}\end{pmatrix}U_m,
$$ where $I_{s_{_{1,m}}}$ are $I_{s_{_{2,m}}}$   the identity matrices of $\mathcal M_{s_{_{1,m}}}(\Bbb
C)$, or $\mathcal M_{s_{_{2,m}}}(\Bbb C)$ respectively.

Note
$$
  x_i\oplus 0 = (I_{\mathcal A}\oplus 0) ( x_i\oplus y_i) (I_{\mathcal A}\oplus
  0)  \in \mathcal A\bigoplus 0.
$$
Thus
$$\begin{aligned}
\psi(x_i\oplus 0) &= [(P_m)_{m=1}^\infty] [(X_i^{(m)})_{m=1}^\infty]
[(P_m)_{m=1}^\infty]\\ &=  [(P_m X_i^{(m)} P_m)_{m=1}^\infty]\\
&=[(U_m^* \begin{pmatrix} I_{s_{_{1,m}}} &   0\\ 0&0\end{pmatrix} U_m X_i^{(m)}
U_m^* \begin{pmatrix} I_{s_{_{1,m}}} &   0\\ 0&0\end{pmatrix} U_m)_{m=1}^\infty].\end{aligned}
$$
Similarly,
$$\begin{aligned}
\psi(0\oplus y_i) &= [(I_{k_m}-P_m)_{m=1}^\infty]
[(X_i^{(m)})_{m=1}^\infty]
[(I_{k_m}-P_m)_{m=1}^\infty]\\ &=  [((I_{k_m}-P_m) X_i^{(m)} (I_{k_m}-P_m))_{m=1}^\infty]\\
&=[(U_m^*\begin{pmatrix} 0 &   0\\ 0&I_{s_{_{2,m}}} \end{pmatrix} U_m X_i^{(m)} U_m^*\begin{pmatrix} 0 &   0\\ 0&I_{s_{_{2,m}}} \end{pmatrix} U_m)_{m=1}^\infty].\end{aligned}
$$
Let $$ \begin{pmatrix} A_i^{(m)} &   0\\ 0&0\end{pmatrix}=\begin{pmatrix} I_{s_{_{1,m}}} &   0\\ 0&0\end{pmatrix} U_m X_i^{(m)}
U_m^*\begin{pmatrix} I_{s_{_{1,m}}} &   0\\ 0&0\end{pmatrix}, \ \text { for } i=1,\ldots,n $$ and $$ \begin{pmatrix} 0 &   0\\ 0& B_i^{(m)}\end{pmatrix}
= \begin{pmatrix} 0 &   0\\ 0&I_{s_{_{2,m}}} \end{pmatrix}   U_m X_i^{(m)} U_m^*\begin{pmatrix} 0 &   0\\ 0&I_{s_{_{2,m}}} \end{pmatrix}, \ \text { for } i=1,\ldots,n
$$ where   $A_1^{(m)},\ldots, A_n^{(m)}$ are in $\mathcal M_{s_{1,m}}(\Bbb C)$ and $B_1^{(m)},\ldots,
B_n^{(m)}$  are in $\mathcal M_{s_{2,m}}(\Bbb C)$.
Then,
$$
\begin{aligned}
 &[(\begin{pmatrix} A_i^{(m)} &   0\\ 0&0\end{pmatrix})_{m=1}^\infty] &=[(U_m)_{m=1}^\infty]\psi(x_i\oplus
0)[(U_m^* )_{m=1}^\infty] \ \text { for } i=1,\ldots,n \\
&[(\begin{pmatrix} 0 &   0\\ 0& B_i^{(m)}\end{pmatrix})_{m=1}^\infty] &=[(U_m)_{m=1}^\infty]\psi(0\oplus
y_i)[(U_m^* )_{m=1}^\infty] \ \text { for } i=1,\ldots,n .\end{aligned}
$$
 Therefore, when $m$ is large enough, we have
$$
 \begin{aligned}
       (A_1^{(m)},\ldots, A_n^{(m)}) &\in \Gamma_R^{(top)}(x_1  ,\ldots,x_n ;s_{_{1,m}},\frac 1 {r_0},
      P_1,\ldots,P_{r_0}),\\
        (B_1^{(m)},\ldots, B_n^{(m)}) &\in \Gamma_R^{(top)}(y_1  ,\ldots,y_n ;s_{_{2,m}},\frac 1 {r_0},
      P_1,\ldots,P_{r_0}),
   \end{aligned}
$$
On the other hand,
$$
\begin{aligned} ([(X_1^{(m)})_{m=1}^\infty], &\ldots ,
[(X_1^{(m)})_{m=1}^\infty]) =(\psi(x_1\oplus y_1), \ldots,
\psi(x_n\oplus
y_n ))\\
& = (\psi(x_1\oplus 0)+ \psi(0\oplus y_1), \ldots, \psi(x_n\oplus
0)+ \psi(0\oplus y_n))\\
&=[(U_m^* )_{m=1}^\infty] \left
([(\begin{pmatrix} A_1^{(m)} &   0\\ 0&0\end{pmatrix})_{m=1}^\infty]+[(\begin{pmatrix} 0 &   0\\ 0& B_1^{(m)}\end{pmatrix})_{m=1}^\infty],
\ldots,\right .\\ & \qquad \qquad \qquad \qquad\left .[(\begin{pmatrix} A_n^{(m)} &   0\\ 0&0\end{pmatrix})_{m=1}^\infty]+[(\begin{pmatrix} 0 &   0\\ 0& B_n^{(m)}\end{pmatrix})_{m=1}^\infty] \right
)
[(U_m)_{m=1}^\infty]\\
&=[(U_m^* )_{m=1}^\infty] \left ([(\begin{pmatrix} A_1^{(m)} &   0\\ 0&B_1^{(m)}\end{pmatrix})_{m=1}^\infty], \ldots,[(\begin{pmatrix} A_n^{(m)} &   0\\ 0&B_n^{(m)}\end{pmatrix})_{m=1}^\infty] \right )
[(U_m)_{m=1}^\infty]\\
%&= ([(U_m  (A_1^{(m)}\oplus B_1^{(m)})U_m^*)_{m=1}^\infty],
%\ldots,[(U_m (A_n^{(m)}\oplus B_n^{(m)})U_m^*)_{m=1}^\infty] ),
\end{aligned}
$$ which is against the inequality (4.1.1). This completes the proof.

\end{proof}

\begin{lemma}
Let $\alpha, \beta>0$ and
$$
    f(s) = \alpha s^2+ \beta (1-s)^2+ 1- s^2-(1-s)^2,\qquad \text{
    for } 0\le s\le 1.
$$ Then
$$
\max_{0\le s\le 1} f(s) = \left \{ \begin{aligned}  &
\frac{\alpha\beta-1}{\alpha+\beta-2}  &\qquad \text { if } \alpha<1,
\beta<1 \\& \max \{\alpha, \beta\} &\qquad \text { otherwise. }
\\\end{aligned} \right.
$$
\end{lemma}
\begin{proof}Note that
$$
   f(s) =  (\alpha+\beta-2) s^2-  {2(\beta-1)}
   s +\beta.
$$ Thus, if $\alpha+\beta\ne 2$, then $f$ has an extreme point at
$$
 s_0=\frac {\beta-1}{\alpha+\beta-2},
$$ with
$$
f(s_0)=\frac{\alpha\beta-1}{\alpha+\beta-2}.
$$

Case one: If $\alpha+\beta> 2$, we know
$$
\frac{\alpha\beta-1}{\alpha+\beta-2}=
\alpha-\frac{\alpha^2-2\alpha+1}{\alpha+\beta-2}\le \alpha=f(1);
\quad \text{ similarly }\quad
\frac{\alpha\beta-1}{\alpha+\beta-2}\le \beta=f(0).
$$
Thus
$$
\max_{0\le s\le 1} f(s) =\max \{\alpha,\beta\} \quad \text { if }
\quad \alpha+\beta> 2.$$

Case two: If $\alpha+\beta-2< 0$ and $f$ achieves its absolute
maximum in the interval $(0,1)$, then $0<s_0<1$. This is equivalent
to
$$
\alpha<1\qquad \text { and } \qquad \beta<1.
$$
Thus
$$
f(s_0)=\frac{\alpha\beta-1}{\alpha+\beta-2}=
\alpha-\frac{\alpha^2-2\alpha+1}{\alpha+\beta-2}\ge \alpha=f(1),
\quad \text{ and }\quad
f(s_0)=\frac{\alpha\beta-1}{\alpha+\beta-2}\ge \beta=f(0).
$$ It follows that
$$
\max_{0\le s\le 1} f(s) = \left \{ \begin{aligned}  &
\frac{\alpha\beta-1}{\alpha+\beta-2}  & \text { if } \alpha<1,
\beta<1\qquad \quad \\& \max \{\alpha, \beta\} &\qquad \text { if }
\alpha+\beta<2, \  \alpha\ge 1 \ \text { or } \alpha+\beta<2, \
\beta\ge 1.
\\ \end{aligned} \right.$$

Case three: If $\alpha+\beta-2= 0$, it is easy to check that
$$
\max_{0\le s\le 1} f(s) =\max\{\alpha, \beta\}.$$

As a summary, we obtain
$$
\max_{0\le s\le 1} f(s) = \left \{ \begin{aligned}  &
\frac{\alpha\beta-1}{\alpha+\beta-2} &\qquad \text { if } \alpha<1,
\beta<1\\ & \max \{\alpha, \beta\} &\qquad \text { otherwise. }
\\\end{aligned} \right.
$$
\end{proof}

\begin{proposition}
Suppose that $\mathcal A$ and $\mathcal B$ are two unital C$^*$
algebras and $x_1\oplus y_1, \ldots, x_n\oplus y_n$ is a family of
self-adjoint elements that generate  $\mathcal A\bigoplus \mathcal
B$. If
$$
\alpha>\delta_{top}(x_1,\ldots,x_n) \qquad \text { and } \qquad
\beta>\delta_{top}(y_1,\ldots,y_n),
$$
then
$$
\delta_{top}(x_1\oplus y_1, \ldots, x_n\oplus y_n)\le  \left \{
\begin{aligned}  & \frac{\alpha\beta-1}{\alpha+\beta-2}  &\qquad
\text { if } \alpha<1, \beta<1 \\& \max \{\alpha, \beta\} &\qquad
\text { otherwise. }
\\\end{aligned} \right.
$$
\end{proposition}

\begin{proof}Let $R>\max\{\|x_1\oplus y_1\|,\ldots, \|x_n\oplus
y_n\|\}$ be a positive number. By Lemma 4.1, there is some
$\omega_0>0$ so that the following hold: for any
$0<\omega<\omega_0$, there are $r(\omega)\in \Bbb N$ and
$K(r(\omega))\in \Bbb N$ satisfying
   \begin{align}
       & { \nu_{\infty}(\Gamma_R^{(top)}(x_1,\ldots,x_n;k_1,\frac 1 {r(\omega)},
      P_1,\ldots,P_{r(\omega)}),\omega) } <\left (  \frac 1 \omega\right)^{\alpha{ k_1^2}}, \quad \forall \ k_1\ge K( r(\omega));\\
         &   { \nu_{\infty}(\Gamma_R^{(top)}(y_1,\ldots,y_n;k_2,\frac 1 {r(\omega)},
      P_1,\ldots,P_{r(\omega)}),\omega)} <\left (  \frac 1 \omega\right)^{\beta{k_2^2}}, \quad \forall \ k_2\ge
      K( r(\omega)).
   \end{align}
On the other hand, for each $0<\omega<\omega_0$ and $r(\omega)\in
\Bbb N$, it follows from Lemma 4.2 that there is some $t\in \Bbb N$
so that $\forall \ r>t$, $\forall \ k\ge 1$, if
$$
  (X_1,\ldots, X_n) \in \Gamma_R^{(top)}(x_1\oplus y_1,\ldots,x_n\oplus y_n;k,\frac 1 r,
      P_1,\ldots,P_r),
$$ then there are
$$
   \begin{aligned}
       (A_1,\ldots, A_n) &\in \Gamma_R^{(top)}(x_1  ,\ldots,x_n ;k_1,\frac 1 {r_\omega},
      P_1,\ldots,P_{r_\omega}),\\
        (B_1,\ldots, B_n) &\in \Gamma_R^{(top)}(y_1  ,\ldots,y_n ;k_2,\frac 1 {r_\omega},
      P_1,\ldots,P_{r_\omega})
   \end{aligned}
$$ and $U\in \mathcal U(k)$ so that (i) $k_1+k_2=k$; and (ii)
$$
  \left \| (X_1,\ldots, X_n) -U^* ( \begin{pmatrix} A_1 &   0\\ 0&B_1\end{pmatrix},\ldots, \begin{pmatrix} A_n &   0\\ 0&B_n\end{pmatrix}    )     U        \right
  \|< \omega.
$$
Moreover, we can further assume that $U\in \mathcal U(k)/ (\mathcal
U(k_1)\bigoplus \mathcal U(k_2))$.

Now it is a standard argument to show that for $r>t$,
 \begin{align}
\nu_\infty(\Gamma_R^{(top)}&(x_1\oplus y_1,\ldots,x_n\oplus
y_n;k,\frac 1 r,
      P_1,\ldots,P_r), 3\omega)\notag\\
   & \le \sum_{k_1+k_2=k}\left (\left (   \frac
   {C_2}{\omega}\right)^{k^2-k_1^2-k_2^2}\cdot \
   \nu_{\infty}(\Gamma_R^{(top)}(x_1  ,\ldots,x_n ;k_1,\frac 1 {r_\omega},
      P_1,\ldots,P_{r_\omega}),\omega)\right .\notag\\
      &  \left .\qquad \qquad \qquad \qquad \cdot \ \nu_\infty( \Gamma_R^{(top)}(y_1  ,\ldots,y_n ;k_2,\frac 1 {r_\omega},
      P_1,\ldots,P_{r_\omega}),\omega)\right ),
      \end{align}  where $C_2$ is some constant independent of $k,\omega$.
But  \begin{align}
    (4.1.4) & = \left ( \sum_{k_1=1}^{K(r(\omega))}+\sum_{k_1=K(r(\omega))+1}^{k-K(r(\omega))-1}+\sum_{k_1=k-K(r(\omega))}^k\right)\left (\left (   \frac
   {C_2}{\omega}\right)^{k^2-k_1^2-(k-k_1)^2}\right .\notag\\ &\qquad \qquad \qquad \qquad \cdot \
   \nu_{\infty}(\Gamma_R^{(top)}(x_1  ,\ldots,x_n ;k_1,\frac 1 {r_\omega},
      P_1,\ldots,P_{r_\omega}),\omega)\notag\\
      &  \left .\qquad \qquad \qquad \qquad \cdot \ \nu_\infty( \Gamma_R^{(top)}(y_1  ,\ldots,y_n ;k-k_1,\frac 1 {r_\omega},
      P_1,\ldots,P_{r_\omega}),\omega)\right ).
\end{align}
Let
$$
   \begin{aligned}
      M_\omega & = \max_{1\le k_1\le K(r(\omega))}\nu_{\infty}(\Gamma_R^{(top)}(x_1  ,\ldots,x_n ;k_1,\frac 1 {r_\omega},
      P_1,\ldots,P_{r_\omega}),\omega)\\
      N_\omega & = \max_{1\le k_2\le K(r(\omega))}\nu_\infty( \Gamma_R^{(top)}(y_1  ,\ldots,y_n ;k_2,\frac 1 {r_\omega},
      P_1,\ldots,P_{r_\omega}),\omega)
   \end{aligned}
$$
By (4.1.2) and (4.1.3), we get that if $k>2K(r(\omega))$ then
\begin{align}
\nu_\infty(\Gamma_R^{(top)}&(x_1\oplus y_1,\ldots,x_n\oplus
y_n;k,\frac 1 r,
      P_1,\ldots,P_r), 3\omega)\notag\\
    &  \le K(r(\omega)) M_\omega \left (\frac {C_2}{\omega}\right)^{k^2-(k-K(r(\omega)))^2}\left (  \frac 1
    \omega\right)^{\beta{k^2}}+K(r(\omega)) N_\omega \left (\frac {C_2}{\omega}\right)^{k^2-(k-K(r(\omega)))^2}\left (  \frac 1
    \omega\right)^{\alpha{k^2}}\notag\\
    &  \qquad \qquad + \ \sum_{k_1=K(r(\omega))+1}^{k-K(r(\omega))-1}
\left(\frac    {C_2}{\omega}\right)^{k^2-k_1^2-(k-k_1)^2} \left (
\frac 1 \omega\right)^{\alpha{ k_1^2}}\left (  \frac 1
\omega\right)^{\beta{(k-k_1)^2}}.
\end{align}
Let $$
    f(s) = \alpha s^2+ \beta (1-t)^2+ 1- s^2-(1-s)^2,\qquad \text{
    for } 0\le s\le 1.
$$ And
$$
L(\alpha,\beta) =\max_{0\le s\le 1} f(s).$$ Then
\begin{align}
(4.1.6) &\le \left[K(r(\omega))( M_\omega+ N_\omega)\left (\frac
{C_2}{\omega}\right)^{k^2-(k-K(r(\omega)))^2}+k C_2^{k^2}\right]\notag\\
&\qquad \qquad \qquad \qquad \cdot   \left\{ \left (  \frac 1
    \omega\right)^{\beta{k^2}} + \left (  \frac 1
    \omega\right)^{\alpha{k^2}}+ \left (  \frac 1
    \omega\right)^{L(\alpha,\beta)k^2 }\right \}.
\end{align}
Note that
$$
\lim_{k\rightarrow \infty}\frac {\log\left[K(r(\omega))( M_\omega+
N_\omega)\left (\frac
{C_2}{\omega}\right)^{k^2-(k-K(r(\omega)))^2}+k
C_2^{k^2}\right]}{k^2}=\log C_2;
$$ and
$$
L(\alpha,\beta)\ge \max\{\alpha,\beta\}.
$$ We obtain,
$$
  \begin{aligned}
    \limsup_{k\rightarrow\infty}
    \frac{\log(\nu_\infty(\Gamma_R^{(top)}(x_1\oplus y_1,\ldots,x_n\oplus
y_n;k,\frac 1 r,  P_1,\ldots,P_r),  3\omega))}{k^2}\le \log C_2+
{L(\alpha,\beta)}\log\left ( \frac 1
    \omega\right).
  \end{aligned}
$$
It induces that
$$
\begin{aligned}
   \delta_{top}&(x_1\oplus y_1,\ldots, x_n\oplus y_n) \\
   & = \limsup_{\omega\rightarrow 0^+} \inf_{r\in \Bbb N}
   \limsup_{k\rightarrow\infty} \frac{\log(\nu_\infty(\Gamma_R^{(top)}(x_1\oplus y_1,\ldots,x_n\oplus
y_n;k,\frac 1 r,  P_1,\ldots,P_r),  \omega))}{-k^2\log\omega}\\
&\quad \\
   &\le L(\alpha,\beta)=\left \{
\begin{aligned}  & \frac{\alpha\beta-1}{\alpha+\beta-2}  &\qquad
\text { if } \alpha<1, \beta<1 \\& \max \{\alpha, \beta\} &\qquad
\text { otherwise, }
\\\end{aligned} \right.
\end{aligned}
$$ where the last equation is from Lemma 4.3.
\end{proof}

\begin{proposition}
Suppose that $\mathcal A$ and $\mathcal B$ are two unital C$^*$
algebras and $x_1\oplus y_1, \ldots, x_n\oplus y_n$ is a family of
self-adjoint elements that generate $\mathcal A\bigoplus \mathcal
B$. If
$$
s=\delta_{top}(x_1,\ldots,x_n) \qquad \text { and } \qquad
t=\delta_{top}(y_1,\ldots,y_n),
$$
then
$$
\delta_{top}(x_1\oplus y_1, \ldots, x_n\oplus y_n)\le  \left \{
\begin{aligned}  & \frac{st-1}{s+t-2}  &\qquad
\text { if } s<1, \ t<1 ;\\& \max \{s,t\} &\qquad \text {
otherwise.}
\\\end{aligned} \right.
$$
\end{proposition}

\begin{proof}
It follows directly from the preceding lemma.

\end{proof}

\subsection{One of topological free entropy dimensions $\ge 1$}

\begin{lemma}
Suppose that $\mathcal A$ and $\mathcal B$ are two unital C$^*$
algebras and $x_1\oplus y_1, \ldots, x_n\oplus y_n$ is a family of
self-adjoint elements that generate  $\mathcal A\bigoplus \mathcal
B$.  Then
$$
\delta_{top} (x_1\oplus y_1, \ldots, x_n\oplus y_n)\ge \max
\{\delta_{top}(x_1,\ldots,x_n), \delta_{top}(y_1,\ldots, y_n) \}.
$$
\end{lemma}

\begin{proof}Let $R>\max\{\|x_1\oplus y_1\|,\ldots, \|x_n\oplus
y_n\|\}$ be a positive number. For any $r\ge 1$, $\epsilon>0$,
$k>k_1$, and any
$$
   \begin{aligned}
       (A_1,\ldots, A_n) &\in \Gamma_R^{(top)}(x_1  ,\ldots,x_n ;k_1,\epsilon,
      P_1,\ldots,P_{r}),\\
        (B_1,\ldots, B_n) &\in \Gamma_R^{(top)}(y_1  ,\ldots,y_n ;k-k_1,\epsilon,
      P_1,\ldots,P_{r})
   \end{aligned}
$$ we have
$$
   ( \begin{pmatrix} A_1  &  0\\ 0&B_1\end{pmatrix},\ldots, \begin{pmatrix} A_n &   0\\ 0&B_n\end{pmatrix}    ) \in \Gamma_R^{(top)}(x_1\oplus y_1, \ldots, x_n\oplus
y_n;k,\epsilon, P_1,\ldots,P_r). $$
Thus,
$$
\nu_{\infty} (\Gamma_R^{(top)}(x_1\oplus y_1, \ldots, x_n\oplus
y_n;k,\epsilon, P_1,\ldots,P_r),\omega)
\ge\nu_\infty(\Gamma_R^{(top)}(y_1,\ldots,y_n;k-k_1,\epsilon,
P_1,\ldots,P_r),2\omega).
$$
It follows that
$$
\delta_{top}(x_1\oplus y_1, \ldots, x_n\oplus
y_n)\ge\delta_{top}(y_1,\ldots,y_n).
$$Similarly,
$$
\delta_{top}(x_1\oplus y_1, \ldots, x_n\oplus
y_n)\ge\delta_{top}(x_1,\ldots,x_n).
$$ Hence we have proved the result of the lemma.
\end{proof}

\begin{theorem}
Suppose that $\mathcal A$ and $\mathcal B$ are two unital C$^*$
algebras and $x_1\oplus y_1, \ldots, x_n\oplus y_n$ is a family of
self-adjoint elements that generates $\mathcal A\bigoplus \mathcal
B$. If one of $\delta_{top}(x_1,\ldots,x_n)$ and
$\delta_{top}(y_1,\ldots,y_n)$ is larger than or equal to $1$, then
$$
\delta_{top}(x_1\oplus y_1, \ldots, x_n\oplus y_n)=\max
\{\delta_{top}(x_1,\ldots,x_n), \delta_{top}(y_1,\ldots,y_n)\}
$$

\end{theorem}
\begin{proof}
The result follows directly from Proposition 4.2 and Lemma 4.4.

\end{proof}

%\begin{remark}
%Combining the results in \cite {HaSh2} and Theorem 4.1, we are able
%to obtain the topological free entropy dimension of some unital
%C$^*$ algebras.
%\end{remark}

\subsection{Both of topological free entropy dimensions $<1$}

We start this subsection with the following definition.
\begin{definition}
Suppose that $\mathcal A$ is a unital C$^*$ algebra and $x_1,\ldots,
x_n$ is a family of self-adjoint elements in $\mathcal A$. The
family of elements $x_1,\ldots,x_n$ is called stable if for any
$\alpha<\delta_{top}(x_1,\ldots,x_n)$ there are positive numbers
$C_3>0$ and $\omega_0>0$, $r_0\in \Bbb N$, $k_0\in\Bbb N$ so that
$$
\nu_\infty(\Gamma_{R}^{(top)}(x_1,\ldots, x_n;q\cdot k_0,\frac 1 r,
P_1,\ldots, P_r), \omega)\ge C_3^{ (q\cdot k_0)^2}\left( \frac 1
\omega\right)^{\alpha \cdot  (q\cdot k_0)^2}, \forall \
0<\omega<\omega_0, r>r_0, q\in \Bbb N.
$$
\end{definition}
\begin{example} \begin{enumerate} \item From the inequality (3.3.1),   it follows that any family of
self-adjoint generators $ x_1,\ldots, x_n $ of $\mathcal M_n(\Bbb
C)$   is
 stable.

\item A  self-adjoint element $x$ in a unital C$^*$ algebra is stable (see
\cite {HaSh}).

 \item Suppose that $\mathcal K$ is the algebra of all compact operators
 in a separable Hilbert space $H$ and unital C$^*$ algebra $\mathcal A$ is the
 unitization of $\mathcal K$. Then any family of self-adjoint
 generators $x_1,\ldots, x_n$ of $\mathcal A$ is stable since $\delta_{top}(x_1,\ldots, x_n)=0$ (see Theorem 5.6 in \cite {HaSh2}). \end{enumerate}
\end{example}

\begin{notation}
 Suppose that $A\in \mathcal M_{k_1}(\Bbb C)$ and $B\in \mathcal M_{k_2}(\Bbb C)$. We denote the element
  $$\left ( \begin{aligned}
  A  & \quad 0\\
  0 &\quad B
  \end{aligned} \right )\in \mathcal M_{k_1+k_2}(\Bbb C)$$ by $A\oplus B$.
  \end{notation} \begin{notation}  Suppose that $\Gamma_1\subset \mathcal (M_{k_1}(\Bbb C))^n$ and $\Gamma_2\subset (\mathcal M_{k_2}(\Bbb C))^n$. We denote the set
  $$\left \{\left ( \begin{aligned}
  A_1 & \quad 0\\
  0 &\quad B_1
  \end{aligned} \right ), \ldots, \left ( \begin{aligned}
  A_n & \quad 0\\
  0 &\quad B_n
  \end{aligned} \right )  \ | \ (A_1,\ldots,A_n)\in \Gamma_1, \ (B_1,\ldots,B_n)\in \Gamma_2\right \}$$ in $( \mathcal M_{k_1+k_2}(\Bbb C))^n$ by $\Gamma_1\oplus \Gamma_2$.
  \end{notation}
The main goal of this subsection is to prove the following result.

\begin{proposition}   Suppose that $\mathcal A$ and
$\mathcal B$ are two unital C$^*$ algebras and $x_1\oplus y_1,
\ldots, x_n\oplus y_n$ is a family of self-adjoint elements that
generates $\mathcal A\bigoplus \mathcal B$ as a C$^*$ algebra.
Assume
$$
s=\delta_{top}(x_1,\ldots,x_n)<1 \qquad \text { and } \qquad
t=\delta_{top}(y_1,\ldots,y_n)<1.
$$ If  both families $\{x_1,\ldots,x_n\}$ and $\{y_1,\ldots,
y_n\}$ are stable, then
$$
\delta_{top}(x_1\oplus y_1, \ldots, x_n\oplus y_n)=
\frac{st-1}{s+t-2}  .
$$ Moreover,   the family of elements
$x_1\oplus y_1, \ldots, x_n\oplus y_n$ is also stable.

\end{proposition}
\begin{remark}
The difficulty to prove the preceding result lies in the fact that
$I_{\mathcal A}\oplus 0$ might not be in the $*$-algebra generated
by $x_1\oplus y_1, \ldots, x_n\oplus y_n$.
\end{remark}

 The proof of Proposition 4.3 will be postponed after we
 prove
 some lemmas firstly.
Recall the definition of the packing number of a set in a metric
space as follows. \begin{definition} Suppose that $X$ is a metric
space with a metric distance $d$. The packing number of a set $K$ by
$\delta$-balls in $X$, denoted by $Pack_d(K,\delta)$, is the maximal
cardinality of the subsets $F $ in $K$ satisfying for all $a,b$ in
$F$ either $a=b$ or $d(a,b)\ge \delta$.
\end{definition}

The following result follows easily from the definition of packing
number.

\begin{lemma}For any subset $K$ of $((\mathcal M_{k}(\Bbb C))^n, \|\cdot \|)$, we have
$$
Pack_\infty(K,\delta)\ge \nu_\infty(K,2\delta)\ge
Pack_\infty(K,4\delta),
$$ where $
Pack_\infty(K,\delta)$ is the packing number of the set $K$ by $\delta$-$\| \ \|$-balls in $(\mathcal M_{k}(\Bbb C))^n$.
\end{lemma}

\begin{lemma}
Let $$  \begin{aligned}\Gamma_1 \subset (\mathcal M_{k_1}(\Bbb C))^n
\qquad \qquad \Gamma_2 \subset (\mathcal M_{k_2}(\Bbb C))^n.
\end{aligned}$$
Then, for $\delta>0$,
$$
Pack_\infty(\Gamma_1\oplus\Gamma_2,\delta) \ge
\nu_\infty(\Gamma_1,2\delta)\cdot \nu_\infty(\Gamma_2, 2\delta),
$$where $\Gamma_1\oplus\Gamma_2$ is as in Notation 4.2.
\end{lemma}
\begin{proof}
    By Lemma 4.5,  there exists a family of elements
    $\{(A_1^\lambda, \ldots,A_n^\lambda) \}_{\lambda\in\Lambda}$, or  $\{(B_1^\sigma, \ldots,B_n^\sigma)\}_{\sigma\in\Sigma}$, in $\Gamma_1$, or in $\Gamma_2$ respectively,
     such that
    $$\begin{aligned} &\|(A_1^{\lambda_1}, \ldots,A_n^{\lambda_1})-(A_1^{\lambda_2}, \ldots,A_n^{\lambda_2})\|\ge \delta, \qquad \forall \
    \lambda_1\ne \lambda_2\in\Lambda\\
    & \|(B_1^{\sigma_1}, \ldots,B_n^{\sigma_1})-(B_1^{\sigma_2}, \ldots,B_n^{\sigma_2})\|\ge\delta, \qquad \forall
    \ \sigma_1\ne \sigma_2\in \Sigma;\end{aligned}
    $$
    and
    $$
     Card(\Lambda)\ge \nu_\infty(\Gamma_1,2\delta),\qquad
     Card(\Sigma)\ge \nu_\infty(\Gamma_2, 2\delta),
    $$ where $Card(\Lambda)$, or $Card(\Sigma)$, is the cardinality
    of the set $\Lambda$, or $\Sigma$ respectively.
    Thus, if $\lambda_1\ne \lambda_2$ or $\sigma_1\ne \sigma_2$,
    $$\begin{aligned}
 &\left  \|\left(\left ( \begin{aligned}
  A_1^{\lambda_1}  & \quad 0\\
  0 &\quad B_1^{\sigma_1}
  \end{aligned} \right ),\ldots,\left ( \begin{aligned}
  A_n^{\lambda_1}  & \quad 0\\
  0 &\quad B_n^{\sigma_1}
  \end{aligned} \right )\right) -\left(\left ( \begin{aligned}
  A_1^{\lambda_2}  & \quad 0\\
  0 &\quad B_1^{\sigma_2}
  \end{aligned} \right ),\ldots,\left ( \begin{aligned}
  A_n^{\lambda_2}  & \quad 0\\
  0 &\quad B_n^{\sigma_2}
  \end{aligned} \right )\right) \right \|\\&\quad = \max \{\|(A_1^{\lambda_1}, \ldots,A_n^{\lambda_1})-(A_1^{\lambda_2}, \ldots,A_n^{\lambda_2})\|,
\|(B_1^{\sigma_1}, \ldots,B_n^{\sigma_1})-(B_1^{\sigma_2}, \ldots,B_n^{\sigma_2})\| \}\\ &\quad \ge\delta.\end{aligned}
    $$ Hence
    $$
Pack_\infty(\Gamma_1\oplus\Gamma_2,\delta) \ge
\nu_\infty(\Gamma_1,2\delta)\cdot \nu_\infty(\Gamma_2, 2\delta).
$$
\end{proof}

\begin{definition}Let $k_1,k_2,s$ be some positive integers such that
$k_1\ge 2s$, $k_2\ge 2s$. Define $\Omega(k_1,k_2,s)$ be the
collection of all these $k_1\times k_2$ matrices $T$ satisfying
$\|T\|\le 2$ and $rank(T)\le 2s$, where $rank(T)$ is the rank of the
matrix $T$.
\end{definition}

\begin{sublemma} Let $k, k_1,k_2,s$ be some positive integers such that
$k_1\ge 2s$, $k_2\ge 2s$ and $k=k_1+k_2$. Let
   $\iota$ be the embedding of $\mathcal M_{k_1,k_2}(\Bbb C)$ into
   $\mathcal M_k(\Bbb C)$ by the mapping
   $$
\iota: \ \     A \rightarrow \begin{pmatrix}
     0 & A\\
     0  & 0
     \end{pmatrix}
   $$ for any $A$ in $\mathcal M_{k_1,k_2}(\Bbb C)$ .
For any $\delta>0$, we have
$$
  \nu_2(\iota(\Omega(k_1,k_2,s)), \delta) \le \left  (  \frac {C_4}{\delta}
  \right)^{4s(k_1+k_2)+2s},
$$ where $C_4$ is a constant independent of $k_1,k_2$ and $s$.
\end{sublemma}

\begin{proof}
For any $T$ in $\Omega(k_1,k_2,s),$ by Definition 4.3 we have $ \|T\|\le 2$ and $
rank(T) \le 2s$. Thus by polar decomposition, there are partial isometry $V_1$  in $\mathcal
M_{k_1,k_2}(\Bbb C)$, a unitary matrix   $V_2$ in $\mathcal
M_{k_2}(\Bbb C)$ and a family of numbers $0\le \lambda_1,\ldots,
\lambda_{2s}\le 2$ such that,
$$\begin{aligned}
T= &V_1 \ diag(\lambda_1,\ldots, \lambda_{2s}, 0,\ldots, 0) \ V_2^*\\
=& (V_1 (I_{2s}\oplus 0\cdot I_{k_2-2s}))\ diag(\lambda_1,\ldots,
\lambda_{2s}, 0,\ldots, 0)\ (V_2 (I_{2s}\oplus 0\cdot
I_{k_2-2s}))^*.\end{aligned}
$$
Now it is a standard argument (for example see \cite{Sza}) to show that
$$
  \nu_2(\iota(\Omega(k_1,k_2,s)), \delta) \le \left  (  \frac {C_4}{\delta}
  \right)^{4s(k_1+k_2)+2s},
$$where $C_4$ is a constant independent of $k_1,k_2$ and $s$.
\end{proof}

 Let $s_1,s_2,s_3$ be positive integers so that $s_1\ge s_3, s_2\ge s_3$.

\begin{definition}
Define $R(s_1,s_3)$ be the collection of all these self-adjoint
matrices $Q$ in $\mathcal M_{s_1+s_3}(\Bbb C)$ satisfying: there are
some unitary matrix $U_1$ in $\mathcal M_{s_1+s_3}(\Bbb C)$ and real
numbers $$\lambda_1,\ldots, \lambda_{s_1}, \ldots,
\lambda_{s_1+s_3}$$ such that (i)
$$
Q= U_1^* diag (\lambda_1,\ldots, \lambda_{s_1},\ldots,
\lambda_{s_1+s_3})U_1;
$$ and (ii) $$  \lambda_i\ge 2, \qquad \forall \ 1\le i\le s_1.  $$

Define $Q(s_2,s_3)$ be the collection of all these self-adjoint
matrices $Q$ in $\mathcal M_{s_2+s_3}(\Bbb C)$ satisfying: there are
some unitary matrix $U_2$ in $\mathcal M_{s_2+s_3}(\Bbb C)$ and real
numbers $$\mu_1,\ldots, \mu_{s_2},\ldots,
\mu_{s_2+s_3}$$ such that (i)
$$
Q= U_2^* diag (\mu_1,\ldots, \mu_{s_2},\ldots,
\mu_{s_2+s_3})U_2;
$$ and (ii) $$  |\lambda_i|\le 1, \qquad \forall \ 1\le i\le s_2.  $$

\end{definition}

\begin{sublemma}Let $\delta>0$ be a positive number.
 Let $s_1,s_2,s_3$ be positive integers so that $s_1\ge s_3, s_2\ge s_3$.
 Let $k_1=s_1+s_3$, $k_2=s_2+ s_3$
and $k=k_1+k_2$. Suppose  $X$ is a $k_1\times k_2$ complex matrix
such that, (i) $\|X\|\le 1$; and (ii) for some $Q_1$ in $R(s_1,s_3)$
and $Q_2$ in $Q(s_2,s_3)$,
$$
  \frac {Tr((Q_1X-XQ_2)^*(Q_1X-XQ_2))}{k}\le \delta.
$$
Then, there is some $T$ in $\Omega(k_1,k_2,s_3)$ (as defined in
Definition 4.3) such that
$$
\frac {Tr(( X-T)^*(X-T))}{k}\le \delta.
$$

\end{sublemma}

\begin{proof}
By the definitions of $R(s_1,s_3)$ and   $Q(s_2,s_3)$, we know there
are some unitary matrix $U_1$ in $\mathcal U(k_1)$, $U_2$ in
$\mathcal U(k_2)$, and families of real numbers $\lambda_1,\ldots,
\lambda_{k_1}$ and $\mu_1,\ldots, \mu_{k_2}$ such that (i)
$$ \begin{aligned}
 Q_1& = U_1^*diag(\lambda_1,\ldots,
\lambda_{k_1}) U_1\\ Q_2& = U_2^*diag(\mu_1,\ldots, \mu_{k_2}) U_2;
\end{aligned}
$$ and (ii)
$$
\lambda_i\ge 2, \  |\mu_j|\le 1, \quad \forall \ 1\le i\le s_1,\
1\le j\le s_2.
$$

Let
$$
   U_1XU_2^*= \left ( \begin{aligned}
    Y_{11}   & \quad Y_{12}\quad \\
    Y_{21}  & \quad Y_{22} \quad
   \end{aligned}\right )\in \ \mathcal M_{k_1,k_2}(\Bbb C),
$$
where $Y_{11}\in \mathcal M_{s_1,s_2}(\Bbb C)$, $Y_{12}\in \mathcal
M_{s_1,s_3}(\Bbb C)$, $Y_{21}\in \mathcal M_{s_3,s_2}(\Bbb C)$ and
$Y_{22}\in \mathcal M_{s_3,s_3}(\Bbb C)$.

From the facts that
$$
  \frac {Tr((Q_1X-XQ_2)^*(Q_1X-XQ_2))}{k}\le \delta,
$$and
$$
\lambda_i\ge 2, \  |\mu_j|\le 1, \quad \forall \ 1\le i\le s_1,\
1\le j\le s_2,
$$
we know that
$$
  \frac {Tr(Y_{11}^* Y_{11})}{k}\le \delta.
$$

 Let
$$
T_1 = \left ( \begin{aligned}
   0\quad    & \quad Y_{12}\quad \\
    Y_{21}  & \quad Y_{22} \quad
   \end{aligned}\right ).
$$
Then $\|T_1\|\le 2\|X\|\le 2, $   $rank(T_1)\le 2s_3 $, and $$\frac
{Tr(( X-U_1^*TU_2)^*(X-U_1^*TU_2))}{k}=\frac {Tr(Y_{11}^*
Y_{11})}{k}\le \delta.
$$
Let $T=U_1^*T_1U_2$ and we finished the proof of the sublemma.

\end{proof}

\begin{lemma}
 Let $s_1,s_2,s_3$ be positive integers so that $s_1\ge s_3, s_2\ge s_3$, and $R(s_1,s_3),    Q(s_2,s_3)$ be defined in
Definition 4.4.  Let $k_1=s_1+s_3$,  $k_2=s_2+s_3$ and
$k=k_1+k_2$.  Then there exists a family of unitary matrices $\{
   U_\gamma\}_{\gamma\in\mathcal I}$ in $\mathcal M_k(\Bbb C)$ so that (i)
   when $\gamma_1\ne\gamma_2\in \mathcal I$,
   $$
    \|U_{\gamma_1}^*(Q_1\oplus Q_2)U_{\gamma_1}-U_{\gamma_2}^*(\tilde Q_1\oplus \tilde Q_2)U_{\gamma_2}\|\ge
    \delta, \  \ \forall \ Q_1,\tilde Q_1\in R(s_1,s_3),   \ Q_2,\tilde Q_2\in Q(s_2,s_3);
   $$  and (ii)
$$
 Card (\mathcal I) \ge  (C_6\cdot 130\delta )^{-k^2}\cdot
\left( \frac {C_5}{\delta} \right )^{-(s_1^2+s_2^2+ 8s_3+
12(k_1s_3+k_2s_3))},
$$ where $C_5,C_6$ are some constants independent of
$k,s_1,s_2,s_3$;  and   $Q_1\oplus Q_2$, $\tilde Q_1\oplus \tilde
    Q_2$ are defined as in Notation 4.1.

\end{lemma}

As we will see, Lemma 4.7 is a   consequence of the Sublemma 4.3.3,
Sublemma 4.3.4 and Sublemma 4.3.5, which we will prove first. Following the
notations as before, let $s_1,s_2,s_3$, $k_1=s_1+s_3$,
$k_2=s_2+s_3$, and $k=k_1+k_2$ be  as above.

\begin{definition}

Define $\mathcal S(s_1,s_2,s_3 )$ to be the collection of all these
matrices
$$
  S= \left (\begin{aligned}  S_{11} & \quad   S_{12} \\
     S_{21}  &\quad  S_{22}
   \end{aligned}\right )
 \in \mathcal M_{k}(\Bbb C), $$ where $S_{ij}\in \mathcal
   M_{k_i,k_j}(\Bbb C)$ for $1\le i,j\le 2$, satisfying
   (i) $\|S_{i,j}\| \le 2$ for $1\le i,j\le 2$; (ii)   $S_{12}\in\Omega(k_1,k_2,s_3)$
  and  $S_{21}\in\Omega(k_2,k_1,s_3)$, where $ \Omega(k_1,k_2,s_3)$
  and  $ \Omega(k_2,k_1,s_3)$ are defined in Definition 4.3.

\end{definition}

\begin{sublemma}
  Suppose that $\delta>0$ and $U_1, U_2$ are unitary matrices in $\mathcal M_k(\Bbb C)$ so that the following
  holds:  there are some $Q_1,\tilde Q_1\in R(s_1,s_3),$ and $ Q_2,\tilde Q_2\in
  Q(s_2,s_3)$ such that
 $$
    \|U_{ 1}^*(Q_1\oplus Q_2)U_{ 1}-U_{ 2}^*(\tilde Q_1\oplus \tilde
    Q_2)U_{ 2}\|\le \delta,
   $$ where $R(s_1,s_3)$, $ Q(s_2,s_3)$ are defined in Definition 4.4 and  $Q_1\oplus Q_2$, $\tilde Q_1\oplus \tilde
    Q_2$ are as in Notation 4.1. Then, there is some $S$ in $\mathcal S(s_1,s_2,s_3)$ such that
$$
  \|U_1-U_2S\|_2\le 2 \delta.
$$
\end{sublemma}

\begin{proof}
Let
$$
U_2U_1^*= \left (
  \begin{aligned}
     U_{11}  & \quad U_{12}\\
     U_{21}  & \quad U_{22}
  \end{aligned}
  \right )
$$ where $U_{ij}$ is $k_i\times k_{j}$ complex matrix for $1\le
i,j\le 2$. By the conditions on $U_1, U_2$, we know that
$\|U_{12}\|\le 1$ and
$$
   \frac{Tr((U_{12}Q_2-\tilde Q_1U_{12})^*(U_{12}Q_2-\tilde
   Q_1U_{12}))}{k}\le \delta^2.
$$ By Sublemma 4.3.2, we know that there is some $T_{12}$ in
$\Omega(k_1,k_2,s_3)$ so that
$$
   \frac{Tr((U_{12}-T_{12})^*(U_{12}-T_{12}))}{k}\le \delta^2.
$$

Similarly, there is some $T_{21}$ in $\Omega(k_2,k_1,s_3)$ so that
$$
   \frac{Tr((U_{21}-T_{21})^*(U_{21}-T_{21}))}{k}\le \delta^2.
$$

 Let
$$
S= \left (
  \begin{aligned}
     U_{11}  & \quad T_{12}\\
     T_{21}  & \quad U_{22}
  \end{aligned}
  \right )
$$ be in $\mathcal M_{k}(\Bbb C)$. Now it is not hard to check that
$S$ is in $\mathcal S(s_1,s_2,s_3)$ and
$$
  \|U_1-U_2S\|_2\le 2 \delta.
$$

\end{proof}

%\begin{sublemma}
%   Suppose that $\|U_1^*DU_1-U_2^*DU_2\|_2\le \delta $ for some
%   unitary matrices $U_1,U_2$ in $\mathcal U(k)$. Then there is some
%   matrix $$
%  S= \left (\begin{aligned}  S_{11} & \quad 0 \quad  & S_{13} \\
%     0 \quad & \quad S_{22} \quad & S_{23} \\
%     S_{31}  &\quad S_{23}\quad  & S_{33}
%   \end{aligned}\right )
%   $$ in $\mathcal M_{k}(\Bbb C)$, where $S_{ij}\in \mathcal
%   M_{s_i,s_j}(\Bbb C)$ for $1\le i,j\le 3$, such that
%   (i) $\|S_{i,j}\| \le 1$ for $1\le i,j\le 3$; (ii) $\|S\|_2\le 1$; (iii)
%  $ \|U_2U_1^*-S\|_2\le 2\delta$.

%\end{sublemma}

\begin{sublemma} For any $\delta>0,$ let
$$
\mathcal S_\delta(s_1,s_2,s_3) = \{ S\in \mathcal S(s_1,s_2,s_3) \ |
\ \exists \ U\in \mathcal U(k) \text { such that } \|U-S\|_2\le
2\delta\}.
$$
   We have$$\nu_2(\mathcal S_\delta(s_1,s_2,s_3), 64\delta)\le \left( \frac {C_5}{\delta} \right
)^{s_1^2+s_2^2+ 8s_3+ 12(k_1 +k_2)s_3},
$$ where $C_5$ is some constant independent of $k, s_1,s_2,s_3$.
\end{sublemma}
\begin{proof}
Assume $$
  S= \left (\begin{aligned}  S_{11} & \quad   S_{12} \\
     S_{21}  &\quad  S_{22}
   \end{aligned}\right )
 , $$ is in $\mathcal S_\delta(s_1,s_2,s_3)$, where $S_{ij}$ is $k_i\times k_j$  complex matrix for $1\le i,j\le 2$.

Assume that
 $$
   S_{11}=H_{11}W_{11}
 $$ is the polar decompositions of elements $S_{11}$  in
 $\mathcal M_{k_1}(\Bbb C)$, where $W_{11}  $ is unitary
 matrix in  $\mathcal M_{k_1}(\Bbb C)$ and $H_{11}$ is a positive  matrix in  $\mathcal M_{k_1}(\Bbb C)$. From the fact that
$\|U-S\|_2\le  2\delta$, it follows that $$
   \begin{aligned}
   (16\delta)^2&\ge(\|(S-U)S^*\|_2+\|U(S- U)^*\|_2)^2\ge\|SS^*-I_k\|_2^2 \\ &\ge
   \frac {Tr ((H_{11}^2- (I_{k_1}-S_{12}S_{12}^*))^2)}
   k.
   \end{aligned}
$$
Let $$2\ge \lambda_1 \ge \lambda_2 \ge \cdots \ge \lambda_{k_1}\ge 0,$$ be the
eigenvalues of $H_{11}$ in $\mathcal M_{k_1}(\Bbb C)$ arranged in
the decreasing order. Note that $S_{12}$ is in
$\Omega(k_1,k_2,s_3)$. By the Definition 4.3, $S_{12}$ is a
$k_1\times k_2$ complex matrix satisfying $\|S_{12}\|\le 2$ and
$rank(S_{12})\le 2s_3$. We can assume that
$$4\ge \mu_1\ge  \mu_2\ge \cdots, \ge \mu_{2s_3} \ge  0   \ge \cdots \ge 0$$
are eigenvalues of $S_{12}S_{12}^*$ in $\mathcal M_{k_1}(\Bbb C)$
arranged in the decreasing order.
 By Lemma 4.1 in
\cite{V2}, we have
$$
   k(16\delta)^2\ge  {Tr ((H_{11}^2- (I_{k_1}-S_{12}S_{12}^*))^2)}
    \ge \sum_{i=1}^{k_1-2s_3}|\lambda_i^2-1|^2+
  \sum_{i=k_1-2s_3+1}^{k_1} |\lambda_i^2+\mu_i-1|^2\ge
  \sum_{i=1}^{k_1-2s_3}|\lambda_i-1|^2.
$$
Thus,  there is some $$U_{11} \in \mathcal U(k_1)/(\mathcal
U(k_1-2s_3)\oplus I_{2s_3})$$ such that
\begin{align}
   \frac {Tr(H_{11}-U_{11}^* diag(1,1,\ldots, 1,
   \lambda_{k_1-2s_3+1},\ldots, \lambda_{k_1})U_{11} )^2}{k}=\frac{ \sum_{i=1}^{k_1-2s_3}|\lambda_i-1|^2}{k} \le
   (16\delta)^2.
\end{align}

Similarly, assume that $$ S_{22}=H_{22}W_{22}
$$ is the polar decomposition of $S_{22}$ in $\mathcal M_{k_2}(\Bbb
C)$ where $W_{22}$ is a unitary matrix and $H_{22}$ is a positive
matrix in $\mathcal M_{k_2}(\Bbb C)$. Then there are some
$$U_{22} \in \mathcal U(k_2)/(\mathcal U(k_2-2s_3)\oplus I_{2s_3})$$ and some $0\le \sigma_{k_2-2s_3+1},\ldots,
\sigma_{k_2}\le 2$ such that
\begin{align}
   \frac {Tr(H_{22}-U_{22}^* diag(1,1,\ldots, 1,
   \sigma_{k_2-2s_3+1},\ldots, \sigma_{k_2})U_{22})^2}{k} \le
   (16\delta)^2.
\end{align}
Define the mapping $\rho$ from the space $$\begin{aligned} &\mathcal X= ( \
\mathcal U(k_1), \frac {\|\cdot\|_{Tr}}{\sqrt k})\  \times \
(\mathcal U(k_2), \frac {\|\cdot\|_{Tr}}{\sqrt k})\   \times \
(\mathcal U(k_1)/(\mathcal U(k_1-2s_3)\oplus  I_{2s_3}), \frac
{\|\cdot\|_{Tr}}{\sqrt k})\ \\
&\qquad \times  \ (\mathcal U(k_2)/(\mathcal U(k_2-2s_3)\oplus
I_{2s_3}), \frac {\|\cdot\|_{Tr}}{\sqrt k}) \  \times\
(\Omega(k_1,k_2,s_3),\frac {\|\cdot\|_{Tr}}{\sqrt k}) \ \times \
(\Omega(k_2,k_1,s_3),\frac {\|\cdot\|_{Tr}}{\sqrt
k})  \\
& \qquad \times \ \{( \lambda_{k_1-2s_3+1},\ldots, \lambda_{k_1})\ |
\ 0\le \lambda_j\le 2, \forall \ k_1-2s_3+1\le j\le k_1\} \\ &
\qquad \times \ \{( \sigma_{k_2-2s_3+1},\ldots, \sigma_{k_2})\ | \
0\le \sigma_j\le 2, \forall \ k_2-2s_3+1\le j\le k_2\}\end{aligned}
$$ into $\mathcal S$ by sending $$
\begin{aligned} (W_{11},&W_{22}, U_{11},U_{22},
   S_{12}, S_{21},(
   \lambda_{k_1-2s_3+1},\ldots, \lambda_{k_1}), (   \sigma_{k_2-2s_3+1},\ldots, \sigma_{k_2}))\in \mathcal X \end{aligned}
$$  to
$$
\tilde S= \left (\begin{aligned}  U_{11}^*\tilde H_{11}U_{11}W_{11}   & \qquad \qquad S_{12} \\
     S_{21}\quad\qquad   & \quad U_{22}^*\tilde H_{22}U_{22}W_{22} \quad
     \quad
   \end{aligned}\right ) \ \in  \mathcal M_k(\Bbb C),
$$ where
$$
   \begin{aligned}
  \tilde H_1 &=  diag(1,1,\ldots, 1,
   \lambda_{k_1-2s_3+1},\ldots, \lambda_{k_1}) \qquad
  \tilde H_2&= diag(1,1,\ldots, 1,
   \sigma_{k_2-2s_3+1},\ldots, \sigma_{k_2}).
   \end{aligned}
$$
  By inequalities (4.3.1) and  (4.3.2), we know for any $S\in\mathcal S_\delta$, there is some $x\in\mathcal X$ satisfying
  $$
    \|S-\rho(x)\|_2\le 16\sqrt 2 \delta.
  $$Computing the covering number of $\rho(\mathcal X)$ by combining with   Sublemma 4.3.1 and Lemma 2.1, we get
$$\begin{aligned}
\nu_2(\mathcal S_\delta(s_1,s_2,s_3), 60\delta) &\le \left( \frac
{C_5}{\delta} \right )^{k_1^2+k_2^2+
k_1^2-(k_1-2s_3)^2+k_2^2-(k_2-2s_3)^2+ 8(k_1 +k_2)s_3 +4s_3+
4s_3}\\&\le \left( \frac {C_5}{\delta} \right )^{k_1^2+k_2^2+ 8s_3+
12(k_1s_3+k_2s_3)},\end{aligned}
$$ where $C_5$ is some constant independent of $k, s_1,s_2,s_3$.

\end{proof}

\begin{sublemma}
   For every $U\in \mathcal U(k)$, let
$$
\Sigma(U)= \{W\in \mathcal U(k)  \ | \ \exists \ S  \in \mathcal
S(s_1,s_2,s_3) \text{ such that }   \ \|W- US\|_2\le   { 2 \delta}
\}.
$$  Then
 the volume of
$\Sigma(U)$ is bounded by the following:
$$
\mu(\Sigma(U)) \le (C_6\cdot 130\delta)^{k^2}\cdot  \left( \frac
{C_5}{\delta} \right )^{s_1^2+s_2^2+ 8s_3+ 12(k_1s_3+k_2s_3)},
$$ where $\mu$ is the normalized Haar measure on the unitary group $\mathcal U(k)$ and $C_5,C_6$ are some constants independent of
$k,s_1,s_2,s_3$.

\end{sublemma}
\begin{proof} \ For any $\delta>0,$ let
$$
\mathcal S_\delta(s_1,s_2,s_3) = \{ S\in \mathcal S(s_1,s_2,s_3) \ |
\ \exists \ U\in \mathcal U(k) \text { such that } \|U-S\|_2\le
2\delta\}.
$$ It follows from the preceding sublemma that
$$\begin{aligned}
\nu_2(\mathcal S_\delta(s_1,s_2,s_3), 60\delta) &\le  \left( \frac
{C_5}{\delta} \right )^{s_1^2+s_2^2+ 8s_3+ 12(k_1s_3+k_2s_3)}
\end{aligned}
$$ where $C_5$ is a   constant independent of $s_1,s_2,s_3$. Thus, by Lemma 2.1, the covering number of the set $\Sigma(U)$ by the $130\delta $-$\|\cdot\|_2$-balls
 in $\mathcal M_k(\Bbb C)$  is bounded by
$$
\nu_2(\Sigma(U),   {130\delta} ) \le \left( \frac {C_5}{\delta}
\right )^{s_1^2+s_2^2+ 8s_3+ 12(k_1s_3+k_2s_3)}.
$$
But the ball of radius $130\delta $ in $\mathcal U(k)$ has a volume
bounded by
$$
\mu(\text {ball of radius $130\delta $ in $\mathcal U(k)$})\le (C_6\cdot 130\delta
)^{k^2},
$$ where $C_6$ is a universal constant.
Thus
$$
\mu(\Sigma(U)) \le  (C_6\cdot 130\delta )^{k^2}\cdot  \left( \frac
{C_5}{\delta} \right )^{s_1^2+s_2^2+ 8s_3+ 12(k_1s_3+k_2s_3)}.
$$
\end{proof}

\begin{proof} [Proof of Lemma 4.7:]  \
For every $U\in \mathcal U(k)$, define
$$
\Sigma(U)= \{W\in \mathcal U(k)  \ | \ \exists S\in \mathcal S(s_1,s_2,s_3),
\text{ such that }   \|W- US\|_2\le 2{ { \delta} }\}.
$$
By previous lemma, we have
$$
\mu(\Sigma(U)) \le (C_6\cdot 130\delta )^{k^2}\cdot  \left( \frac
{C_5}{\delta} \right )^{s_1^2+s_2^2+ 8s_3+ 12(k_1s_3+k_2s_3)}.
$$
 A ``parking" (or  exhausting) argument will show
the existence of a family of unitary elements $\{ U_i
\}_{i=1}^N\subset \mathcal U(k)$ such that
$$
N\ge (C_6\cdot 130\delta )^{-k^2}\cdot  \left( \frac {C_5}{\delta}
\right )^{-(s_1^2+s_2^2+ 8s_3+ 12(k_1s_3+k_2s_3))}.
$$  and
$$
  U_i \  \text { is not contained in } \cup_{j=1}^{i-1}\Sigma (U_j).
$$
Hence $$ \| U_i-U_jS\|_2\ge   { { 2\delta} } ,\qquad \forall \ S\in
\mathcal S(s_1,s_2,s_3), \text{ with }   \ \forall 1\le j<i \le N.
$$
By Sublemma 4.3.3, we know that for all $1\le j<i \le N$
$$
    \|U_{i}^*(Q_1\oplus Q_2)U_{i}-U_{j}^*(\tilde Q_1\oplus \tilde Q_2)U_{j}\|\ge
    \delta, \  \ \forall \ Q_1,\tilde Q_1\in R(s_1,s_3),   \ Q_2,\tilde Q_2\in Q(s_2,s_3);
   $$
i.e.  there exists a family of unitary matrices $\{
   U_\gamma\}_{\gamma\in\mathcal I}$ in $\mathcal M_k(\Bbb C)$ so that (i)
   when $\gamma_1\ne\gamma_2\in \mathcal I$,
    $$
    \|U_{\gamma_1}^*(Q_1\oplus Q_2)U_{\gamma_1}-U_{\gamma_2}^*(\tilde Q_1\oplus \tilde Q_2)U_{\gamma_2}\|\ge
    \delta, \  \ \forall \ Q_1,\tilde Q_1\in R(s_1,s_3),   \ Q_2,\tilde Q_2\in Q(s_2,s_3);
   $$   and (ii)
$$
 Card (\mathcal I) \ge  (C_6\cdot 130\delta )^{-k^2}\cdot
\left( \frac {C_5}{\delta} \right )^{-(s_1^2+s_2^2+ 8s_3+
12(k_1s_3+k_2s_3))},
$$ where $C_5,C_6$ are some constants independent of
$k,s_1,s_2,s_3$ and $R(s_1,s_3),     Q(s_2,s_3)$ are defined in
Definition 4.4.
\end{proof}

\begin{lemma}Let $ k_1, m\ge 2  $ be some positive integers.
   Suppose that $Q$ is a self-adjoint element in $\mathcal M_{k_1}^{s.a.}(\Bbb
   C)$ such that
   $$
       \|Q-3I_{k_1} \|<\frac 2 {m^3},
   $$ where $I_{k_1}$ is the identity matrix of $\mathcal
   M_{k_1}(\Bbb C)$. Then     $Q $ is in $  R( k_1-\frac
   {4k_1}{m^4},  \frac
   {4k_1}{m^4})$, where $ R( k_1-\frac
   {4k_1}{m^4},  \frac
   {4k_1}{m^4})$ is defined in
Definition 4.4.
\end{lemma}
\begin{proof}
Suppose that $\lambda_1\ge \lambda_2\ge \ldots\ge \lambda_{k_1}$ are
the eigenvalues of $Q$. Let
$$\begin{aligned}
  T_1&= \{i \in \Bbb N\ | 1\le i\le k_1 \text { and }
|\lambda_i-3|\le \frac 1 m \}
\end{aligned}$$ and $$
T_2=\{1,2,\ldots,k_1 \}\setminus T_1.
$$ By
Lemma 4.1 in \cite {V2}, we have
$$\begin{aligned}
 k_1\left (\frac 2 {m^3}\right)^2&\ge  Tr((Q-  3I_{k_1} )^2)\ge \sum_{i\ \in \{1,\ldots,k_1\}\setminus T_1} |\lambda_i-3|^2  \ge \left
  ( \frac 1m  \right)^2 \ card(  T_2),\end{aligned}
$$  where $card(  T_2)$ is the cardinality of the set $  T_2$. Thus
$$
card (  T_2) \le \frac {4k_1}{m^4}.
$$Hence, by Definition 4.3, we have   $Q$ is in $R( k_1-\frac
   {4k_1}{m^4},  \frac
   {4k_1}{m^4})$.

\end{proof}

Similarly, we have the following result.

\begin{lemma}Let $ k_2, m\ge 2  $ be some positive integers.
   Suppose that $Q$ is a self-adjoint element in $\mathcal M_{k_2}^{s.a.}(\Bbb
   C)$ such that
   $$
       \|Q \|<\frac 2 {m^3}.
   $$   Then  $Q$ is in $Q( k_2-\frac
   {4k_2}{m^4},  \frac
   {4k_2}{m^4})$, where $Q( k_2-\frac
   {4k_2}{m^4},  \frac
   {4k_2}{m^4})$ is defined in
Definition 4.4.
\end{lemma}

%\begin{lemma}
%Suppose that $\mathcal A$ and $\mathcal B$ are two unital C$^*$
%algebras and $x_1\oplus y_1, \ldots, x_n\oplus y_n$ is a family of
%self-adjoint elements that generates $\mathcal A\bigoplus \mathcal
%B$. For each $\omega>0$, there are sequences
%$\{r_q,k_{1,q},k_{2,q}\}_{q=1}^\infty$ in $\Bbb N$ such that
%$$
% \begin{aligned}
%    \lim_{q\rightarrow\infty} \frac{\log(\Gamma_{R}^{(top)}(x_1,\ldots,x_n;k_{1,q},\frac 1 {r_q},P_1,\ldots, P_{r_q}),\omega)}{-k^2_{1,q}\log\omega}
%        &= \delta_{top}(x_1,\ldots, x_n;\omega,R)\\
%        \lim_{q\rightarrow\infty} \frac{\log(\Gamma_{R}^{(top)}(y_1,\ldots,y_n;k_{1,q},\frac 1 {r_q},P_1,\ldots, P_{r_q}),\omega)}{-k^2_{1,q}\log\omega}
%        &= \delta_{top}(y_1,\ldots, y_n;\omega,R)\\
% \end{aligned}
%$$

%\end{lemma}

\begin{lemma}
Suppose that $\mathcal A$ and $\mathcal B$ are two unital C$^*$
algebras and $x_1\oplus y_1, \ldots, x_n\oplus y_n$ is a family of
self-adjoint elements that generates $\mathcal A\bigoplus \mathcal
B$. For $m\ge 2$, choose
$$
z_m=P_m(x_1\oplus y_1, \ldots, x_n\oplus y_n)
$$ to be a self-adjoint element in $\mathcal A\bigoplus \mathcal B$,
where $P_m(x_1\oplus y_1, \ldots, x_n\oplus y_n)$ is a
noncommutative polynomial of $x_1\oplus y_1, \ldots, x_n\oplus y_n$,
satisfying
$$
   \|z_m-3I_{\mathcal A}\oplus 0\|\le \frac 1 {m^3}.
$$
Then
$$
\delta_{top}(x_1\oplus y_1, \ldots, x_n\oplus y_n, z_m )\le
\delta_{top} (x_1\oplus y_1, \ldots, x_n\oplus y_n).
$$
\end{lemma}
\begin{proof}The result can be proved in the similar fashion as
the one of Lemma 5.1 in \cite{HaSh2}.
\end{proof}

Now we are ready to present the proof of Proposition 4.3.
\begin{proof}[{\bf Proof of Proposition 4.3: }]
Let $R>\max\{4, \|x_1\oplus y_1\|,\ldots, \|x_n\oplus y_n\|\}$ be a
positive number.  Since both families of $x_1,\ldots,x_n$ and
$y_1,\ldots,y_n$ are stable, if
$$\begin{aligned}
  \alpha&<\delta_{top}(x_1,\ldots,x_n)\\
   \beta &< \delta_{top}(y_1,\ldots,y_n),\end{aligned}
$$
then there are some constants $C_7>0$ and $\omega_0>0$, $r_0\ge 1$,
$k_1,k_2\ge 1$ so that
\begin{align}
\nu_\infty(\Gamma_{R}^{(top)}(x_1,\ldots, x_n;q\cdot k_1,\frac 1 r,
P_1,\ldots, P_r), \omega)&\ge C_7^{ q\cdot k_1}\left( \frac 1
\omega\right)^{\alpha \cdot  q\cdot k_1}, \forall \omega<\omega_0,
r>r_0, q\in \Bbb N,\\
\nu_\infty(\Gamma_{R}^{(top)}(y_1,\ldots, y_n;q\cdot k_2,\frac 1 r,
P_1,\ldots, P_r), \omega)&\ge C_7^{ q\cdot k_2}\left( \frac 1
\omega\right)^{\beta \cdot  q\cdot k_2}, \forall \omega<\omega_0,
r>r_0, q\in \Bbb N.\end{align}

% For each $\omega>0$, there are sequences
%$\{r_q,k_{1,q},k_{2,q}\}_{q=1}^\infty$  in $\Bbb N$, depending on
%$\omega$, such that
%$$
% \begin{aligned}
%    \lim_{q\rightarrow\infty} \frac{\log(\nu_\infty(\Gamma_{R}^{(top)}(x_1,\ldots,x_n;k_{1,q},\frac 1 {r_q},P_1,\ldots, P_{r_q}),\omega))}{-k^2_{1,q}\log\omega}
%        &= \delta_{top}(x_1,\ldots, x_n;\omega,R)\\
%        \lim_{q\rightarrow\infty} \frac{\log(\nu_\infty(\Gamma_{R}^{(top)}(y_1,\ldots,y_n;k_{2,q},\frac 1 {r_q},P_1,\ldots, P_{r_q}),\omega))}{-k^2_{2,q}\log\omega}
%        &= \delta_{top}(y_1,\ldots, y_n;\omega,R)\\
% \end{aligned}
%$$

For such $\alpha,\beta>0$, define
$$
  f(a) = \alpha a^2+\beta (1-a)^2+ 1-a^2-(1-a)^2, \quad \forall \
  0\le a\le 1.
$$
Since $\alpha<1$ and $\beta<1$, we know
$$
\max_{0\le a \le 1} f(a) =\frac{\alpha\beta-1}{\alpha+\beta-1}.
$$
 For any $\gamma>0$, let $b, c$ be some  positive integers such
that
$$
  f\left (\frac {bk_{1}}{bk_{1}+ck_{2}}\right )>\frac{\alpha\beta-1}{\alpha+\beta-1} -\gamma.
$$

For $m\ge 2$, choose
$$
z_m=Q_m(x_1\oplus y_1, \ldots, x_n\oplus y_n)
$$ to be a self-adjoint element in $\mathcal A\bigoplus \mathcal B$,
where $Q_m(x_1\oplus y_1, \ldots, x_n\oplus y_n)$ is a self-adjoint
noncommutative polynomial of $x_1\oplus y_1, \ldots, x_n\oplus y_n$,
satisfying
$$
   \|z_m-3I_{\mathcal A}\oplus 0\|\le \frac 1 {m^3},
$$
i.e.
$$
\begin{aligned}
 & \|Q_m(x_1,\ldots, x_m)-3I_{\mathcal A}\|\le\frac 1 {m^3};\\
 &  \|Q_m(y_1,\ldots, y_m)\|\le\frac 1 {m^3}.
\end{aligned}
$$

 For any given $r\ge 1$ and $\epsilon>0$, by the definition of
topological free entropy dimension,  there exist $r'\ge r$ and
$\epsilon'<\epsilon$ such that the following hold:% $\forall \ \tilde
%k_1,\tilde k_2\in \Bbb N$, if
%$$  \begin{aligned}
%(\tilde A_1,\ldots, \tilde A_n) &\in \Gamma_R^{(top)}(x_1,\ldots,
%x_n;\tilde k_1,\epsilon',
%P_1,\ldots,P_{r'})  \\
%(\tilde B_1,\ldots,\tilde B_n) &\in \Gamma_R^{(top)}(y_1,\ldots,
%y_n;\tilde k_2,\epsilon', P_1,\ldots,P_{r'}) ,
%\end{aligned}
%$$ then
%$$
%(\tilde A_1\oplus\tilde B_1,\ldots,\tilde A_n\oplus\tilde B_n,
%Q_m(\tilde A_1\oplus \tilde B_1,\ldots,\tilde A_n\oplus \tilde B_n))
%\in \Gamma_R^{(top)}(x_1\oplus y_1,\ldots, x_n\oplus
%y_n,z_m;k,\epsilon, P_1,\ldots,P_{r}),
%$$ where $k=\tilde k_1+\tilde k_2$. It follows that
\ $\forall \   q\in \Bbb N $, if
$$  \begin{aligned}
(A_1,\ldots, A_n) &\in \Gamma_R^{(top)}(x_1,\ldots, x_n; qb k_{1
},\epsilon',
P_1,\ldots,P_{r'})=\Gamma_1 \\
(B_1,\ldots, B_n) &\in \Gamma_R^{(top)}(y_1,\ldots, y_n; qc k_{2
},\epsilon', P_1,\ldots,P_{r'})=\Gamma_2,
\end{aligned}
$$ then
$$
(A_1\oplus B_1,\ldots, A_n\oplus B_n, Q_m(A_1\oplus
B_1,\ldots,A_n\oplus B_n)) \in \Gamma_R^{(top)}(x_1\oplus
y_1,\ldots, x_n\oplus y_n,z_m;k,\epsilon, P_1,\ldots,P_{r}),
$$ where $k= qbk_{1}+ qck_{2}$.

Let
$$\begin{aligned}
 \Omega(\Gamma_1, \Gamma_2)&= \{U^*(A_1\oplus B_1,\ldots, A_n\oplus B_n, Q_m(A_1\oplus
B_1,\ldots,A_n\oplus B_n))U \ | \  \\
 & \qquad \qquad \qquad \qquad U\in \mathcal U(k), \ (A_1,\ldots, A_n) \in \Gamma_1, \ (B_1,\ldots, B_n)\in
 \Gamma_2\}.
\end{aligned}
$$

By   Lemma 4.6, there is a family of elements $\{(A_1^\lambda,
  \ldots, A_n^\lambda )\}_{\lambda\in\Lambda}$, or  $\{(B_1^\sigma,
  \ldots, B_n^\sigma )\}_{\sigma\in\Sigma}$ , in $\Gamma_1$, or $\Gamma_2$ respectively, so
that
\begin{equation}
\|(A_1^\lambda\oplus B_1^\sigma, \ldots, A_n^\lambda\oplus
B_n^\sigma)-(A_1^{\lambda'}\oplus B_1^{\sigma'}, \ldots,
A_n^{\lambda'}\oplus B_n^{\sigma'})\|>\omega,\  \forall \ (\lambda,
\sigma)\ne (\lambda',\sigma') \in\Lambda\times\Sigma; \end{equation}
and
$$
Card(\Lambda)Card(\Sigma)\ge \nu_\infty(\Gamma_1, 2\omega)\cdot
\nu_\infty(\Gamma_2,2\omega).
$$

Note, for any $(\lambda,\sigma)\in\Lambda\times\Sigma$, we have
$$
\begin{aligned}
  &\|Q_m(A_1^\lambda,\ldots, A_n^\lambda)-3I_{qbk_1}\|  \le \|Q_m(x_1 ,\ldots,x_n )-3I_{\mathcal A} \|+ \epsilon<\frac
2
  {m^3}\\
&\|Q_m(B_1^\sigma,\ldots, B_1^\sigma)\| \le \|Q_m( y_1,\ldots, y_n)
\|+ \epsilon<\frac 2
  {m^3}
\end{aligned}
$$
By Lemma 4.8 and Lemma 4.9, we have that \begin{equation}Q_m(A_1^\lambda,\ldots,
A_n^\lambda) \ \in \ R(qbk_1-\frac {4k_1}{m^4},  \frac
{4k_1}{m^4})\subseteq R(qbk_1-\frac {4k}{m^4},  \frac {4k}{m^4})\end{equation}
and \begin{equation} Q_m(B_1^\sigma,\ldots, B_1^\sigma)\ \in \ R(qbk_1-\frac
{4k_2}{m^4}, \frac {4k_2}{m^4})\subseteq Q(qck_2-\frac {4k}{m^4},
 \frac {4k}{m^4}).\end{equation}
On the other hand, from Lemma 4.7, there exists a family of unitary
matrices $\{
   U_\gamma\}_{\gamma\in\mathcal I}$ in $\mathcal M_k(\Bbb C)$ so that (i)
   when $\gamma_1\ne\gamma_2\in \mathcal I$,
    \begin{align}
    &\|U_{\gamma_1}^*(Q_1\oplus Q_2)U_{\gamma_1}-U_{\gamma_2}^*(\tilde Q_1\oplus \tilde Q_2)U_{\gamma_2}\|\ge
    \omega, \notag \\ & \qquad \quad \  \ \forall \ Q_1,\tilde Q_1\in R(qbk_1-\frac {4k}{m^4},  \frac {4k}{m^4}),   \ Q_2,\tilde Q_2\in Q(qck_2-\frac {4k}{m^4},
 \frac {4k}{m^4});\end{align}   and (ii)
$$\begin{aligned}
 Card (\mathcal I) &\ge  (C_6\cdot 130\omega )^{-k^2}\cdot
\left( \frac {C_5}{\omega} \right )^{-((qbk_1-\frac
{4k}{m^4})^2+(qck_2-\frac {4k}{m^4})^2+ 8\frac {4k}{m^4}+ 12(k\frac
{4k}{m^4}))}\\
&\ge   (C_6\cdot 130\omega)^{-k^2}\cdot \left( \frac {C_5}{\omega}
\right )^{-((qbk_1)^2+(qck_2)^2+    \frac {72k^2}{m^4})
}\end{aligned}
$$ where $C_5,C_6$ are some constants independent of
$k,m$.

 Consider the family of matrices
$$
\{U_\gamma^*   ( A_1^\lambda\oplus B_1^\sigma, \ldots,
A_n^\lambda\oplus B_n^\sigma,Q_m(A_1^\lambda\oplus B_1^\sigma,
\ldots, A_n^\lambda\oplus B_n^\sigma))
 U_\gamma\}_{\lambda\in\Lambda, \sigma\in\Sigma, \gamma\in\mathcal I}
$$
in $\Omega(\Gamma_1,\Gamma_2).$ By (4.3.6), (4.3.7) and (4.3.8) we know that, if
$\gamma_1\ne\gamma_2\in \mathcal I$, then for any
$(\lambda_1,\sigma_1)$ and $ (\lambda_2,\sigma_2)  $ in
$\Lambda\times \Sigma$,
$$\begin{aligned}
\|U_{\gamma_1}^* &Q_m(A_1^{\lambda_1}\oplus B_1^{\sigma_1},
\ldots, A_n^{\lambda_1}\oplus B_n^{\sigma_1}) U_{\gamma_1}\\
&  -U_{\gamma_2}^* Q_m(A_1^{\lambda_2}\oplus B_1^{\sigma_2}, \ldots,
A_n^{\lambda_2}\oplus B_n^{\sigma_2}) U_{\gamma_2}\|\ge \omega.
\end{aligned}$$
Combining with (4.3.5), we have
$$\begin{aligned}
Pack_\infty(\Omega(\Gamma_1,\Gamma_2),\omega)&\ge Card(\Lambda)Card(\Sigma) Card (\mathcal I) \\&\ge
\nu_\infty(\Gamma_1, 2\omega)\nu_\infty(\Gamma_2,2\omega)(C_6\cdot
130\omega)^{-k^2}\cdot \left( \frac {C_5}{\omega} \right
)^{-((qbk_1)^2+(qck_2)^2+    \frac {72k^2}{m^4}) }.\end{aligned}
$$
By inequalities (4.3.3) and (4.3.4), when $\omega,\epsilon$ are
small, $\forall \ q\in \Bbb N$  we have
 \begin{align}
 &  { \nu_{\infty}(\Gamma_R^{(top)}(x_1\oplus y_1,\ldots,
x_n\oplus
y_n,z_m;k,\epsilon, P_1,\ldots,P_{r})),\omega/2)} \notag\\
& \qquad \ge Pack_\infty(\Omega(\Gamma_1,\Gamma_2), \omega)\notag\\
%&\qquad  \ge Card(\Lambda)\cdot Card(\Sigma)\cdot Card(\mathcal I)\notag\\
& \qquad  \ge C_7^{(qbk_1)^2}\left( \frac 1 \omega\right)^{\alpha
\cdot (qbk_1)^2}\cdot C_7^{(qck_2)^2}\left( \frac 1
\omega\right)^{\beta \cdot(qck_2)^2}\cdot (130C_6\omega)^{-k^2}\cdot
\left (\frac{ C_5 }{ \omega}\right )^{- (qbk_{1})^2 -(qck_{2})^2-
\frac {72k^2}{m^4} }
\notag\\
&\qquad \ge C_8^{k^2} \left (  \frac 1 \omega \right
)^{(\frac{\alpha\beta-1}{\alpha+\beta-2}-\gamma-\frac
{72}{m^4})k^2},
\end{align}  where $k=qbk_{1}+qck_{2}$ and $C_8$ is a constant independent of $k,\omega$.
Then, it induces that
$$\begin{aligned}
   \limsup_{\omega\rightarrow 0^+}&\inf_{r\in \Bbb N} \limsup_{k\rightarrow
   \infty}\frac {\log( \nu_{\infty}(\Gamma_R^{(top)}(x_1\oplus y_1,\ldots,
x_n\oplus y_n,z_m;k,\epsilon,
P_1,\ldots,P_{r})),\omega))}{-k^2\log\omega}\\
&\qquad \ge \frac{\alpha\beta-1}{\alpha+\beta-2}-\gamma-\frac
{72}{m^4} .\end{aligned}
$$
Since $\gamma, m$ are arbitrary, we  obtain
$$
\delta_{top}(x_1\oplus y_1,\ldots, x_n\oplus y_n,z_m)\ge
\frac{\alpha\beta-1}{\alpha+\beta-2}.
$$
Hence, by Lemma 4.10,
$$
\delta_{top}(x_1\oplus y_1,\ldots, x_n\oplus y_n)\ge
\delta_{top}(x_1\oplus y_1,\ldots, x_n\oplus y_n,z_m)\ge
\frac{st-1}{s+t-2},
$$where
$$
 s=\delta_{top} (x_1,\ldots,x_n) \quad \text { and } \quad t= \delta_{top}
 (y_1,\ldots,y_n).
$$

 (i) Combining with Proposition 4.2, we have that
$$
\delta_{top}(x_1\oplus y_1,\ldots, x_n\oplus y_n)=
\frac{st-1}{s+t-2},
$$ where
$$
 s=\delta_{top} (x_1,\ldots,x_n) \quad \text { and } \quad t= \delta_{top}
 (y_1,\ldots,y_n).
$$

(ii) Moreover, by inequality (4.3.9), we know that $x_1\oplus
y_1,\ldots, x_n\oplus y_n,z_m$ is a stable family. Since $z_m$ is a
polynomial of $x_1\oplus y_1,\ldots, x_n\oplus y_n$, we know that
 $x_1\oplus y_1,\ldots, x_n\oplus y_n$ is also a stable family.

\end{proof}

\begin{remark}
If $I_{\mathcal A}\oplus 0$ is in the $*$-algebra generated by
$x_1\oplus y_1,\ldots, x_n\oplus y_n$, i.e. there is a
non-commutative polynomial $P$ such that $I_{\mathcal A}\oplus 0=
P(x_1\oplus y_1,\ldots, x_n\oplus y_n)$, then   a much simpler proof
can be provided by using Lemma 3.3 in \cite{HaSh2} instead of Lemma
4.7 here.
\end{remark}

\subsection{Conclusion}
As a summary, we have the following result.
\begin{theorem}
Suppose that $\mathcal A$ and $\mathcal B$ are two unital C$^*$
algebras and $x_1\oplus y_1, \ldots, x_n\oplus y_n$ is a family of
self-adjoint elements that generates $\mathcal A\bigoplus \mathcal
B$.  Assume
$$
s=\delta_{top}(x_1,\ldots,x_n)  \qquad \text { and } \qquad
t=\delta_{top}(y_1,\ldots,y_n).
$$

\noindent (i) If $s\ge 1$ or $t\ge 1$, then
$$
\delta_{top}(x_1\oplus y_1, \ldots, x_n\oplus y_n)=\max
\{\delta_{top}(x_1,\ldots,x_n), \delta_{top}(y_1,\ldots,y_n)\}
$$
\noindent (ii)
  If $s<1$,  $t<1$ and both families $\{x_1,\ldots,x_n\}$, $\{y_1,\ldots,
y_n\}$ are stable, then
$$
\delta_{top}(x_1\oplus y_1, \ldots, x_n\oplus y_n)=
\frac{st-1}{s+t-2} ;
$$
and  the family of elements $x_1\oplus y_1, \ldots, x_n\oplus y_n$
is also stable.
\end{theorem}

\section{Topological free entropy dimension of finite dimensional
C$^*$ algebras}

In this section, we are going to compute the topological free
entropy dimension of a family of self-adjoint generators of a finite
dimensional C$^*$ algebra.

\begin{theorem}
Suppose that $\mathcal A$ is a finite dimensional C$^*$ algebra and
$dim_{\Bbb C}\mathcal A$ is the complex dimension of $\mathcal A$.
If $x_1,\ldots,x_n$ is a family of self-adjoint generators of
$\mathcal A$, then
$$
\delta_{top}(x_1,\ldots,x_n)=1-\frac 1 { dim_{\Bbb C}\mathcal A}.
$$
\end{theorem}
\begin{proof}
It is well known that
$$
\mathcal A\backsimeq \mathcal M_{n_1}(\Bbb C) \bigoplus \mathcal
M_{n_2}(\Bbb C) \bigoplus \cdots \bigoplus \mathcal M_{n_m}(\Bbb C),
$$ for a sequence of positive integers $n_1,\ldots,n_m$. By Theorem
3.1 and Theorem 4.2, we have
$$
\delta_{top}(x_1,\ldots,x_n)=1-\frac 1{n_1^2+\cdots+ n_m^2}= 1-\frac
1 { dim_{\Bbb C}\mathcal A}.
$$

 \end{proof}

Similarly, we have the following result.
 \begin{theorem}
Suppose that $\mathcal K$ is the algebra of all compact operators in
a separable Hilbert space $H$. Suppose that $\mathcal A$ is the
unitization of $\mathcal K$ and $\mathcal B$ is a finite dimensional C$^*$ algebra. If $x_1,\ldots, x_m$ is a family of self-adjoint elements that
generates $\mathcal A\bigoplus\mathcal B$ as a C$^*$ algebra, then
$$
\delta_{top}(x_1,\ldots, x_m)= 1- \frac 1 {dim_{\Bbb C}
\mathcal B+1}.
$$
 \end{theorem}

\end{document}